\documentclass[11pt,oneside,reqno]{article}
\pdfoutput=1
\usepackage[top=1in, bottom=1in, left=1in, right=1in]{geometry}
\usepackage{amsmath,amsthm,amssymb}
\usepackage{bbm,tikz}
\usepackage{mathrsfs,mathtools}
\usepackage{enumitem}
\usepackage{color}
\usepackage{wasysym}
\usepackage[authoryear,round]{natbib}
\usepackage{array,soul}
\usepackage[normalem]{ulem}
\newcommand{\stkout}[1]{\ifmmode\text{\sout{\ensuremath{#1}}}\else\sout{#1}\fi}
\newcommand{\ignore}[1]{}
\newcommand{\Psieps}{\Psi_{\varepsilon}}
\usepackage{filemod}
\usepackage{scalerel}
\usepackage{url}

\usepackage{etoolbox} 
\makeatletter
\usepackage{breakcites,microtype}
\def\@noindentfalse{\global\let\if@noindent\iffalse}
\def\@noindenttrue {\global\let\if@noindent\iftrue}
\def\@aftertheorem{%
  \@noindenttrue
  \everypar{%
    \if@noindent%
      \@noindentfalse\clubpenalty\@M\setbox\z@\lastbox%
    \else%
      \clubpenalty \@clubpenalty\everypar{}%
    \fi}}
\theoremstyle{plain}
\newtheorem{theorem}{Theorem}[section]
\AfterEndEnvironment{theorem}{\@aftertheorem}
\newtheorem{proposition}[theorem]{Proposition}
\AfterEndEnvironment{proposition}{\@aftertheorem}
\newtheorem{lemma}[theorem]{Lemma}
\AfterEndEnvironment{lemma}{\@aftertheorem}

\AfterEndEnvironment{corollary}{\@aftertheorem}

\theoremstyle{definition}
\newtheorem{remark}[theorem]{Remark}
\AfterEndEnvironment{remark}{\@aftertheorem}
\newtheorem{definition}[theorem]{Definition}
\AfterEndEnvironment{definition}{\@aftertheorem}

\AfterEndEnvironment{example}{\@aftertheorem}

\usepackage[titles]{tocloft}
\usepackage{todonotes} 
\usepackage{xcolor} 
\usepackage{marginnote}

\newcommand*{\colorboxed}{}
\def\colorboxed#1#{%
  \colorboxedAux{#1}%
}
\newcommand*{\colorboxedAux}[3]{%
  \begingroup
    \colorlet{cb@saved}{.}%
    \color#1{#2}%
    \boxed{%
      \color{cb@saved}%
      #3%
    }%
  \endgroup
}

\usepackage[linktocpage,colorlinks,linkcolor=purple,anchorcolor=blue,citecolor=purple,urlcolor=black,pagebackref]{hyperref}

\renewcommand*{\backref}[1]{\ifx#1\relax \else Page #1 \fi}
\renewcommand*{\backrefalt}[4]{%
  \ifcase #1 \footnotesize{(Not cited.)}%
  \or        \footnotesize{(Cited on page~#2.)}%
  \else      \footnotesize{(Cited on pages~#2.)}%
  \fi
}

\usepackage{setspace}

\def\be#1{\begin{equation*}#1\end{equation*}}
\def\ben#1{\begin{equation}#1\end{equation}}
\def\bes#1{\begin{equation*}\begin{split}#1\end{split}\end{equation*}}
\def\besn#1{\begin{equation}\begin{split}#1\end{split}\end{equation}}

\def\ba#1{\begin{align*}#1\end{align*}}
\def\ban#1{\begin{align}#1\end{align}}

\def\given{\typeout{Command 'given' should only be used within bracket command}}
\newcounter{@bracketlevel}
\def\@bracketfactory#1#2#3#4#5#6{
\expandafter\def\csname#1\endcsname##1{%
\addtocounter{@bracketlevel}{1}%
\global\expandafter\let\csname @middummy\alph{@bracketlevel}\endcsname\given%
\global\def\given{\mskip#5\csname#4\endcsname\vert\mskip#6}\csname#4l\endcsname#2##1\csname#4r\endcsname#3%
\global\expandafter\let\expandafter\given\csname @middummy\alph{@bracketlevel}\endcsname
\addtocounter{@bracketlevel}{-1}}%
}
\def\bracketfactory#1#2#3{%
\@bracketfactory{#1}{#2}{#3}{relax}{1mu plus 0.25mu minus 0.25mu}{0.6mu plus 0.15mu minus 0.15mu}
\@bracketfactory{b#1}{#2}{#3}{big}{1mu plus 0.25mu minus 0.25mu}{0.6mu plus 0.15mu minus 0.15mu}
\@bracketfactory{bb#1}{#2}{#3}{Big}{2.4mu plus 0.8mu minus 0.8mu}{1.8mu plus 0.6mu minus 0.6mu}
\@bracketfactory{bbb#1}{#2}{#3}{bigg}{3.2mu plus 1mu minus 1mu}{2.4mu plus 0.75mu minus 0.75mu}
\@bracketfactory{bbbb#1}{#2}{#3}{Bigg}{4mu plus 1mu minus 1mu}{3mu plus 0.75mu minus 0.75mu}
}
\bracketfactory{clc}{\lbrace}{\rbrace}
\bracketfactory{clr}{(}{)}
\bracketfactory{cls}{[}{]}
\bracketfactory{abs}{\lvert}{\rvert}
\bracketfactory{norm}{\Vert}{\Vert}
\bracketfactory{floor}{\lfloor}{\rfloor}
\bracketfactory{ceil}{\lceil}{\rceil}
\bracketfactory{angle}{\langle}{\rangle}

\newcounter{ctr}\loop\stepcounter{ctr}\edef\X{\@Alph\c@ctr}%
	\expandafter\edef\csname s\X\endcsname{\noexpand\mathscr{\X}}
	\expandafter\edef\csname c\X\endcsname{\noexpand\mathcal{\X}}
	\expandafter\edef\csname b\X\endcsname{\noexpand\boldsymbol{\X}}
	\expandafter\edef\csname I\X\endcsname{\noexpand\mathbbm{\X}}
	\expandafter\edef\csname r\X\endcsname{\noexpand\mathrm{\X}}
\ifnum\thectr<26\repeat

\let\@IE\IE\let\IE\undefined
\newcommand{\IE}{\mathop{{}\@IE}\mathopen{}}
\let\@IP\IP\let\IP\undefined
\newcommand{\IP}{\mathop{{}\@IP}}

\newcount\minute
\newcount\hour
\newcount\hourMins
\def\now{%
\minute=\time%
\hour=\time \divide \hour by 60%
\hourMins=\hour \multiply\hourMins by 60%
\advance\minute by -\hourMins%
\zeroPadTwo{\the\hour}:\zeroPadTwo{\the\minute}%
}
\def\zeroPadTwo#1{\ifnum #1<10 0\fi#1}

\numberwithin{equation}{section}
\allowdisplaybreaks[4]

\renewcommand\section{\@startsection {section}{1}{\z@}%
{-3.5ex \@plus -1ex \@minus -.2ex}%
{1.3ex \@plus.2ex}%
{\center\small\sc\mathversion{bold}\MakeUppercase}}

\def\subsection#1{\@startsection {subsection}{2}{0pt}%
{-3.5ex \@plus -1ex \@minus -.2ex}%
{1ex \@plus.2ex}%
{\bf\mathversion{bold}}{#1}}

\def\subsubsection#1{\@startsection{subsubsection}{3}{0pt}%
{\medskipamount}%
{-10pt}%
{\normalsize\itshape}{\kern-2.2ex. #1.}}

\def\blfootnote{\xdef\@thefnmark{}\@footnotetext}

\makeatother

\def\note#1{\par\smallskip%
\noindent%
\llap{$\boldsymbol\Longrightarrow$}%
\fbox{\vtop{\hsize=0.98\hsize\parindent=0cm\small\rm #1}}%
\rlap{$\boldsymbol\Longleftarrow$}%
\par\smallskip}

\renewcommand{\cite}{\citet}

\def\^#1{\ifmmode {\mathaccent"705E #1} \else {\accent94 #1} \fi}
\def\~#1{\ifmmode {\mathaccent"707E #1} \else {\accent"7E #1} \fi}

\edef\-#1{\noexpand\ifmmode {\noexpand\bar{#1}} \noexpand\else \-#1\noexpand\fi}
\def\>#1{\vec{#1}}
\def\.#1{\dot{#1}}

\def\atop{\@@atop}

\renewcommand{\leq}{\leqslant}
\renewcommand{\geq}{\geqslant}
\renewcommand{\phi}{\varphi}
\newcommand{\eps}{\varepsilon}

\newcommand{\eq}{\eqref}

\newcommand{\bigo}{\mathrm{O}}
\newcommand{\lito}{\mathrm{o}}

\newcommand{\Var}{\mathop{\mathrm{Var}}\nolimits}
\newcommand{\law}{\mathscr{L}}

\def\sp#1{^{(#1)}}

\newcommand{\Cov}{\mathrm{Cov}}

\newcommand{\diam}{\mathrm{diam}}
\newcommand{\dist}{\mathtt{d}}

\newcommand{\Lip}{\mathrm{Lip}}

\def\nr#1{\textcolor{black}{#1}}

\usepackage{fancyhdr}
\begin{document}

\title{\sc\bf\large\MakeUppercase{
Gaussian random field approximation via Stein's method with applications 
to wide random neural networks
}
}

\author{
  \begin{tabular}{c@{\hskip 0.5in}c}
 Krishnakumar Balasubramanian &  Larry Goldstein\\
\it University of California, Davis & \it University of Southern California \\
\texttt{kbala@ucdavis.edu} &\texttt{larry@usc.edu} \\
\\
 Nathan Ross & Adil Salim \\
\it University of Melbourne & \it Microsoft Research\\
\texttt{nathan.ross@unimelb.edu.au} & \texttt{adilsalim@microsoft.com}
\end{tabular}
}
\date{\today}

\maketitle
\begin{abstract}
We derive upper bounds on the Wasserstein distance ($W_1$), with respect to $\sup$-norm, between any continuous $\mathbb{R}^d$ valued random field indexed by the $n$-sphere and the Gaussian, based on Stein's method. We develop a novel Gaussian smoothing technique that allows us to transfer a bound in a smoother metric to the $W_1$ distance. The smoothing is based on covariance functions constructed using powers of Laplacian operators, designed so that the associated Gaussian process has a tractable Cameron-Martin or Reproducing Kernel Hilbert Space. This feature enables us to move beyond one dimensional interval-based index sets that were previously considered in the literature. Specializing our general result, we obtain the first bounds on the Gaussian random field approximation of wide random neural networks of any depth and Lipschitz activation functions at the random field level. Our bounds are explicitly expressed in terms of the widths of the network and moments of the random weights. We also obtain tighter bounds when the activation function has three bounded derivatives. \\

\end{abstract}

\tableofcontents

\section{Introduction}

Random fields that arise in a variety of applications related to deep learning~\citep{Neal1996,lee2018deep,Matthews2018,yang2019wide,Hanin2021} and stochastic optimization~\citep{benveniste2012adaptive,sirignano2020mean,chen2020dynamical,rotskoff2022trainability,balasubramanian2023high} can exhibit limiting Gaussian behavior, rigorously understood through the theory of weak convergence. Combining this asymptotic behavior with the comprehensive theory of Gaussian random fields leads to insights about the qualitative and quantitative behavior of the random field of interest. In order to justify the accuracy of the approximation of quantities of interest by those of their limits, it is important to quantify the error in the Gaussian random field approximation. Indeed, in the standard multivariate central limit theorem,  Berry-Esseen bounds precisely determine when the Gaussian behavior ``kicks-in''. Our main goal in this work is to develop such quantitative Berry-Esseen-type bounds for Gaussian random field approximations via Stein's method. We focus in particular on bounds in the Wasserstein metric ($W_1$) with respect to sup-norm, and highlight that convergence of these bounds to zero implies asymptotic weak convergence.  Moreover, such bounds immediately imply Wasserstein bounds between important statistics of the fields, such as finite-dimensional distributions and extrema.

Stein's method has been extensively developed to provide quantitative distributional approximation bounds in both the Gaussian and non-Gaussian settings; we refer to~\cite*{chen2011normal,ross2011fundamentals,nourdin2012normal} for a detailed treatment of the former. Recent works (see, for example, \citet[Section 1.1]{Barbour2021}) have focused on developing Stein's method to derive Gaussian process approximations results. These works pertain to random process indexed by the interval $[0,T]$, for some $T< \infty$. As is common in Stein's method, bounds are first developed in some ``smooth'' metric and are then transferred to the metric of interest, such as the Wasserstein, L\'evy-Prokhorov or Kolmogorov metrics, via various smoothing techniques.

For instance,~\citet[Lemma 1.10]{Barbour2021} develops an infinite-dimensional
analog of a widely-used finite-dimensional Gaussian smoothing technique. Based on this foundation, the authors establish Gaussian process approximation bounds for processes indexed by the interval $[0,T]$, in the $W_1$ and L\'evy-Prokhorov metrics. However, their smoothing technique is restricted to random processes indexed by some subset of the real line, as  it relies on a detailed understanding of the Cameron-Martin space of one-dimensional Brownian motion. As there are no canonical Gaussian random fields \emph{indexed} by more general sets, e.g., the $n$-sphere, which have  explicit Cameron-Martin spaces, new ideas are required  to adapt these smoothing techniques to this setting.

A main contribution of this work is the development of a novel smoothing technique which can be used in conjunction with Stein's method to derive Gaussian random field approximation bounds in the $W_1$ metric. The smoothing technique is based on the construction of a Gaussian random field with an explicit Cameron-Martin space via Laplacian operators. Though we focus on the case of random fields indexed by the $n$-sphere $\cS^n$, our approach is generally applicable to random fields indexed by any compact metric measure space $\cM$, subject to increased technical complexity.   

We apply our general result to derive quantitative bounds for the $W_1$ distance between the output of a wide random neural network indexed by inputs in $\cS^n$ and the corresponding Gaussian random field. Though wide random neural networks produce highly complicated random fields, such bounds allow them to be studied via their more tractable limiting Gaussian behavior. In the one hidden layer case, \cite{Neal1996} argues that wide random neural networks asymptotically behave as Gaussian random fields. The works of \cite{Matthews2018} and \cite{lee2018deep} give heuristic and empirical evidence that general depth neural networks exhibit Gaussian random field limits. Very recently, \cite{Hanin2021} proves that deep neural networks converge weakly to appropriately defined Gaussian random fields as the layer widths tend to infinity. At a high-level, one proceeds here by first establishing convergence of finite dimensional distributions, which typically follows directly from the multivariate CLT. Weak convergence then follows from tightness results. In a different but related direction, \cite{Li2022} provide a  characterization of the limiting covariance matrix of the output of the neural network when evaluated at a finite-set of points, as the depth and width tends to infinity at the same rate. 

From a quantitative point of view, the question of how wide a random neural network has to be in order that the limiting Gaussian random field provides a good approximation is left unanswered by results that only demonstrate weak convergence. Works that addresses this gap include~\cite{Eldan2021, Basteri2022, Klukowski2022,bordino2023},  discussed in more detail in Section~\ref{sec:comparison}. However, results currently known to us have at least one of the following drawbacks: they (i) work in weaker topologies, such as Wasserstein metrics with respect to integral (e.g., $L^2$) distances, rather than the $\sup$-norm, (ii) only provide approximation bounds for finite dimensional distributions, and not at the random field level, (iii) require Gaussian or similar restrictive assumptions on the random weights, (iv) consider special cases like one hidden-layer neural networks or use restricted activation functions, such as polynomials. In contrast, our work provides precise quantitative bounds for the error in approximating wide random neural networks with Gaussian random fields, without any of the above-mentioned restrictions. 

In the remainder of the introduction, we state and discuss our main results.  Section \ref{sec:master} is devoted to our smoothing result, Theorem~\ref{thm:strongermetric}. Section~\ref{sec:appdnn} contains our Gaussian approximation results for wide random neural networks, Theorems~\ref{thm:w1bound} and~\ref{thm:rateimprovment}.

\begin{table}[t]
    \centering
    \begin{tabular}{|c|c|}
    \hline
      Notation   & Description \\ \hline \hline 
      $(\cM,\dist, \nu)$ & Metric space $(\cM,\dist)$ equipped with a measure $\nu$ \\ \hline
      $\mathcal{S}^n$ & $n$-sphere \\ \hline
      $\mathrm{C}(\cM;\IR^d)$ & Banach space of continuous functions equipped with $\sup$-norm \\ \hline
    $(\eps,\delta)$ & Regularization and smoothing parameters respectively\\ 
       \hline
       $F,H$  & Random fields in $\mathrm{C}(\cM;\IR^d)$\\\hline
       $G$ & Gaussian Random Field used to approximate $F\in \mathrm{C}(\cM;\IR^d)$ \\\hline
       $S$ & Smoothing Gaussian Random Field \\ \hline
      $\mathsf{d}_{\cH}(F,H)$ & Integral probability metric over a class of test functions $\mathcal{H}$\\\hline
      $D^k$ & $k$-th order Fr\'echet derivative\\ 
       \hline
    \end{tabular}
    \caption{Summary of some main notations used.}
\end{table}

\subsection{Bounds for random field approximations}\label{sec:master}
We now formally describe our setting and main result. Consider a compact metric space $(\cM,\dist)$, equipped with a finite Borel measure $\nu$ that is positive on open balls. Let $\mathrm{C}(\cM;\IR^d)$ denote the (separable) Banach space of continuous functions $f:\cM\to \IR^d$, equipped with the $\sup$-norm $\norm{f}_\infty:=\sup_{x\in \cM} \norm{f(x)}_2$, where $\norm{\cdot}_2$ is the usual Euclidean norm in $\IR^d$. For two random fields $F,H \in \rC(\cM;\mathbb{R}^d)$, we are interested in the distributional approximation of the random field $F$ by $H$ in appropriate distances, which we introduce next.
 
For a function $\zeta:\mathrm{C}(\cM;\IR^d)\to \IR$, we denote taking Fr\'echet derivatives by $D, D^2,\ldots$, and let the operator norm $\norm{\cdot}$ be defined for a $k(\geq 1)$-linear form $A$ on $\mathrm{C}(\cM;\IR^d)$
by $\norm{A}:=\sup_{\norm{f}_\infty=1} \abs{A[f,\ldots,f]}$. The (integral) probability distances we consider are given by the supremum of the differences $|\IE \zeta(F) - \IE \zeta(H)|$ taken over all functions in some class~$\cH$ 
of \emph{test functions} that map $\mathrm{C}(\cM;\IR^d) \to \IR$:  
\be{
  \mathsf{d}_{\cH}(F,H) \coloneqq  \sup_{\zeta\in\cH}\babs{\IE\cls{\zeta(F)}- \IE\cls{ \zeta(H)}}. 
}
In particular, we are interested in the case where the role of $\cH$ is played by
\be{
 \cW \coloneqq   
 \left\{ \zeta:\rC(\cM;\IR^d) \rightarrow \mathbb{R} : \sup_{f \not = h}\frac{\abs{\zeta(f)-\zeta(h)}}{\norm{f-h}_\infty}\leq 1 \right\},
}
the class of $1$-Lipschitz functions, in which case the distance is called as the Wasserstein metric ($W_1$), denoted by $\mathsf{d}_\cW(F,H)$; convergence in this metric is known to imply weak convergence in (the Polish space) $\clr{\mathrm{C}(\cM;\IR^d), \norm{\cdot}_\infty}$; see  \cite[Theorem~11.3.3]{Dudley2002}.

To proceed, we introduce the following weaker metric based on the class of ``smooth'' test functions
\begin{align} \label{def:calF}
\cF \coloneqq \bclc{\zeta:\rC(\cM; \IR^d)\to \IR:\, \sup_f\norm{D^k \zeta(f)} \leq 1, k=1,2; 
\, \sup_{f\not=h}\frac{\norm{D^2\zeta(f)-D^2\zeta(h)}}{\norm{f-h}_\infty}\leq\,1 }.
\end{align}
The metric $\mathsf{d}_{\cF}$ is well-suited to Stein's method, but, in contrast to analogous metrics in the finite dimensional case, it does not directly imply weak convergence, or provide bounds on more informative metrics such as the Wasserstein or L\'evy-Prokhorov. Conceptually speaking, this disconnect can occur because it is not established that the test functions in $\cF$ capture tightness, and practically speaking, it can occur because the technical tools used in finite dimensions (approximation by smoother functions and boundary measure inequalities) do not generally directly carry over to infinite dimensions. Using our novel Laplacian-based smoothing method, we non-trivially adapt the techniques of \cite{Barbour2021}, and prove the following general approximation result in the $W_1$ metric for random fields indexed by the sphere.

\begin{theorem}\label{thm:strongermetric}[Master Theorem]
 Let $F,H \in \rC(\mathcal{S}^n;\mathbb{R}^d)$ be random fields, where $\mathcal{S}^n$ is the unit sphere in $\mathbb{R}^{n+1}$ for some finite integer $n$.
Then for any $\eps,\delta \in(0,1)$ and $\iota>0$,
\begin{align}\label{eq:masterresult}
   \mathsf{d}_\cW(F,H) \leq C \Big( d\,\delta^{-2} \eps^{-2(n+\iota)}~\underset{\mathsf{Section}~\ref{sec:firstterm}}{\colorboxed{red}{\mathsf{d}_{\cF}(F,H)}} \,+\,\underset{\mathsf{Section}~\ref{sec:2and3terms}}{\colorboxed{purple}{\IE \norm{F-F_\eps}_\infty}  + \colorboxed{purple}{ \IE \norm{H-H_\eps}_\infty}}\,+\,\underset{\mathsf{Section}~\ref{sec:wass}}{\colorboxed{magenta}{\delta \sqrt{d}}}
   \Big),
\end{align}
where $F_\eps$ and $H_\eps$ are $\eps$-regularizations of $F$ and $H$ defined at~\eq{eq:genrep} below, and $C$ is a constant depending only on $n$ and $\iota$.
\end{theorem}

To explain the terms appearing in the bound, we first give the basic idea behind the proof of Theorem~\ref{thm:strongermetric}. Given a function $\zeta:\rC(\mathcal{S}^n;\IR^d)\to \IR$ which is  Lipschitz, 
we define a $(\eps,\delta)$-regularized version $\zeta_{\eps,\delta}$ such that for $k=1,2$, $D^k \zeta_{\eps,\delta}(f)$ exists and has norm bounded uniformly in $f$ of order smaller than $\delta^{-2}\eps^{-2(n+\iota)}$, and  $D^2\zeta_{\eps,\delta}$ is Lipschitz with respect to the operator norm, with constant of order 
 $\delta^{-2}\eps^{-2(n+\iota)}$. In particular, there is a constant $c$ such that, $c\, \delta^{2}\eps^{2(n+\iota)}\zeta_{\eps,\delta}\in \cF$. Applying the triangle inequality yields
\be{
\babs{\IE\cls{\zeta(F)}-\IE\cls{\zeta(H)}}\leq \babs{\IE\cls{\zeta_{\eps,\delta}(F)}-\IE\cls{\zeta_{\eps,\delta}(H)}}+ \babs{\IE\cls{\zeta(F)}-\IE\cls{\zeta_{\eps,\delta}(F)}}+\babs{\IE\cls{\zeta_{\eps,\delta}(H)}-\IE\cls{\zeta(H)}}.
}
Because $c\, \delta^{2}\eps^{2(n+\iota)}\zeta_{\eps,\delta}\in \cF$, the first term is bounded of order  $\delta^{-2}\eps^{-2(n+\iota)} \times \mathsf{d}_{\cF}$. In Theorem~\ref{thm:stnsmoothbd} we bound  $\mathsf{d}_{\cF}(F,H)$ when $H$ is a continuous and centered $\IR^d$ valued Gaussian random field, denoted by $\clr{G(x)}_{x\in \cM}$, having non-negative definite covariance kernel $C_{ij}(x,y)=\IE[G_i(x)G_j(y)]$. This result follows from a development of Stein's method closely related to that of
\cite{Barbour2021a}, following \cite{barbour1990stein}.  

In contrast to the first term in \eqref{eq:masterresult}, the remaining three terms decay as $\eps$ and $\delta$ become small, and in particular, the second and third terms become small because $\zeta$ and $\zeta_{\eps,\delta}$ become close. The quantity
$\norm{F-F_\eps}_\infty$ is closely related to the modulus of continuity of $F$ (see Definition \ref{defn:modulus}), and hence the term $ \IE \norm{F-F_\eps}_\infty$ can be further bounded using classical quantitative tightness arguments, which we present in Lemma~\ref{lem:momepsreg}. The 
optimal choice of $\eps$ and $\delta$ is the one having the best tradeoff between the first and the remaining terms, which may in applications depend on the rate of decay 
of $\mathsf{d}_{\cF}(F,H)$ as a function of `sample' or `network' size, and which mitigates its prefactor tending to infinity.

While this approach is a standard way to parlay a preliminary bound in a smooth metric into a stronger one, the crux of the problem at the random field level is: 
 \emph{how does one construct $\zeta_{\eps,\delta}$}? 
In finite dimensions, a fruitful regularization takes a function $\zeta$ and replaces it with $\zeta_\delta(x)=\IE[\zeta(x + \delta S)]$, where $S$ is a ``smoothing'' standard Gaussian. The smoothness of $\zeta_\delta$  follows by making a change of measure and using
the smoothness of the Gaussian density. See, for example,~\cite{raivc2018multivariate} and references therein for additional details.

For random fields \emph{indexed} by~$\cM$ (or even $\cS^n$ with $n\geq2$), there is no ``standard'' Gaussian and in choosing an appropriate smoothing Gaussian $S$ there are two related potential difficulties. The first is that  Cameron-Martin change of measure formulas involve Paley-Wiener integrals, which in general do not have closed form expressions. Moreover, the Cameron-Martin (or Reproducing Kernel Hilbert) space where the change of measure formula holds is typically restricted to a strict subset of $\rC(\cM;\IR^d)$, meaning that $\law(f +\delta S)$ and $\law(\delta S)$ will be singular for many reasonable $f$. Following the strategy of  \cite{Barbour2021}, one approach is to  
define a smoothing Gaussian random field
$S:\cM\to \IR^d$, where the Cameron-Martin space is a subset of smooth functions. In the simpler setting of \cite{Barbour2021} where $\cM=[0,T]$, $S$ is taken to be Brownian motion with a random Gaussian initial value, and the Cameron-Martin space is well known to be absolutely continuous functions equipped with $L^2$-derivative inner product. In our more general setting of random fields indexed by $\cM$, there is no canonical Gaussian process like Brownian motion with a well-understood Cameron-Martin space.

In our construction of a smoothing Gaussian random field indexed by $\cS^n$, the associated Cameron-Martin space contains a class of functions in the domain of a certain fractional Laplacian and whose images are $L^2$ bounded, and thus can be equipped with a related $L^2$ inner product. With this function class in hand, there is still the issue that not all functions $f\in \rC(\cS^n;\IR^d)$ are in the domain of a fractional Laplacian, and so we use a second $\eps$-regularization, now of $f$, given by $f_\eps(x)=\IE f(B\sp{x}_\eps)$, where $(B\sp{x}_t)_{t\geq0}$ is a Brownian motion on $\cS^n$ started from $x$. Now defining $\zeta_{\eps,\delta}(f)=\IE[\zeta(f_\eps + \delta S)]$, bounds on derivatives of $\zeta_{\eps,\delta}$ can be derived from quantitative information on the spectrum of the Laplacian, which is available in detail for $\cM=\cS^n$. This procedure is elaborated in Section~\ref{sec:smth}.

Although Theorem~\ref{thm:stnsmoothbd} for bounding $\mathsf{d}_{\cF}(F,G)$ holds for any compact metric measure space $(\cM,\dist,\nu)$, specializing to the case of $\cM=\cS^n$ in Theorem~\ref{thm:strongermetric} allows us to obtain explicit bounds in terms of the problem parameters (i.e., $n$ and $d$, etc.). The technology of our Laplacian-based smoothing approach applies in
\nr{more general settings.
Explicit bounds can be obtained using our approach anytime there are appropriate estimates for the heat kernel (to prove analogs of Lemma~\ref{lem:momepsreg}) and the spectrum  of the Laplacian (to
construct covariance functions analogous to Section~\ref{sec:constructcov}).
For general Riemannian manifolds, such quantitative spectral estimates are well studied; for example, see \cite{Grigoryan2009}, \cite{Grieser2002} and \cite{zelditch2017eigenfunctions}. Even more generally, understanding Gaussian random fields and Laplacian operators on general metric measure spaces is an active area; see, for example, \cite{sturm1998diffusion} and~\cite{burago2019spectral}.} We highlight that our proofs would also work with functionals of the Laplacian other than fractional powers, as long as they would ensure the required smoothness conditions are satisfied. This flexibility in our proof technique might turn out to be crucial in cases when $\cM$ is not the $n$-sphere.

\subsection{Application to wide random neural networks}\label{sec:appdnn}
We now show how Theorem~\ref{thm:strongermetric} is used to obtain quantitative bounds on the distributional approximation of wide random neural networks by appropriately defined Gaussian random fields. Our first motivation to do so is as follows. In practice, widely used training algorithms like stochastic gradient descent are initialized randomly. In light of that, an interesting question was raised by~\cite{golikov2022nongaussian}: \emph{Does the distribution of the initial weights matter for the training process?} The authors demonstrate that for a large class of distributions of the initial weights, wide random neural networks are Gaussian random fields in the limit. Based on this outcome, they argue that as long as the distribution of the weights are from this universality class, the answer to the above question is \emph{no}. Our results in this section could be used to  quantify this phenomenon. 

Our second motivation is to initiate the study of the training dynamics of neural networks for prediction problems, at the random field level. Several works~\citep{sirignano2020mean,chen2020dynamical,rotskoff2022trainability} demonstrate that when neural networks are trained by gradient descent with small order step-sizes, certain functionals exhibit limiting Gaussian behavior along the training trajectory. Under larger order step-sizes, the works~\citep{damian2022neural,ba2022highdimensional,abbe2022merged} demonstrate that neural networks behave differently than Gaussian-process based prediction methods (including certain classes of kernel methods), thus suggesting the existence of a phase transition from Gaussian to non-Gaussian limits. Our result in this section, along with the associated proof techniques, take a first step towards understanding the above phenomena at the random field level, by developing quantitative information about the setting where the Gaussian behavior is observed.

Formally, we consider a fully connected $L$-layer neural network that is defined recursively through random fields 
$F\sp{\ell}: \cM\to \IR^{n_\ell}$, $\ell=1,\ldots,L$, where $n_1, \ldots, n_L$ are positive integers corresponding to the widths of the network, with $n_L$ assumed  constant. We also assume that $\cM\subset\IR^{n_0}$. The random fields are generated by a collection of random matrices $(W\sp\ell)_{\ell=0}^{L-1}$ where $W\sp\ell:\IR^{n_\ell}\to\IR^{n_{\ell+1}}$, 
with 
$W\sp0$ having i.i.d.\ rows, $W^{(\ell)}$ 
 having independent entries for $1 \leq \ell \leq L-1$, and a collection $(b\sp{\ell})_{\ell=0}^{L-1}$ of centered Gaussian ``bias'' vectors. For $x \in \cM$, we define
\ba{
F\sp{1}(x)& = W\sp{0} x + b\sp{0}, \\
F\sp{\ell}(x) &= W\sp{\ell-1} \sigma\bclr{F\sp{\ell-1}(x)} + b\sp{\ell-1}, \,\,\, \ell=2,\ldots, L,
}
where $\sigma:\IR\to\IR$ is an activation function that we apply to vectors coordinate-wise.
We assume that 
\be{
\Var(W_{ij}\sp{\ell})=\frac{c_w\sp{\ell}}{n_\ell}, \,\, \mbox{ and } \Var(b_i\sp{\ell}) = c_b\sp{\ell}.
}

The limiting Gaussian random field is defined inductively as follows. First let $G\sp{1}=F\sp{1}$, which in general is not a Gaussian random field (since $W_{ij}\sp0$ is not assumed Gaussian), and has covariance 
\be{
C_{ij}\sp{1}(x,y)=\delta_{ij} \bbbclr{\frac{c_w\sp{0}}{n_0}\angle{x,y} + c_b\sp{0}},
}
where $\delta_{ij}$ is the Kronecker delta, and $\angle{\cdot, \cdot}$ is the usual Euclidean inner product.  
Given the distribution of $G\sp{\ell}$ for some $\ell\geq 1$, we define $G\sp{\ell+1}$ to be a centered Gaussian random field with covariance 
\be{
C_{ij}\sp{\ell+1}(x,y) = \delta_{ij}\bbbclr{c_w\sp{\ell}\IE\bbcls{\sigma\bclr{G\sp{\ell}_1(x)} \sigma\bclr{G\sp{\ell}_1(y)}} +c_b\sp{\ell}}.
}
As the rows of $W\sp\ell$ are assumed i.i.d.\ and the network is fully connected, the components of $F^{(\ell)}, \ell \geq 1$ are exchangeable, so in particular identically distributed. Additionally, the 
covariance functions of $F^{(\ell+1)}$ obey the same recurrence as the one above, with $G_1\sp{\ell}$ replaced by $F_1\sp{\ell}$, and hence have uncorrelated components. 
Consequently, paralleling the covariance structure of $F^{(\ell)}$, the components of the 
Gaussian weighted network $G\sp{\ell}$ are, additionally, made independent.

We now state our results for neural networks that have Lipschitz activation function. Widely used activation functions such as the $\mathsf{ReLU}$, $\mathsf{sigmoid}$, $\mathsf{softmax}$, and $\mathsf{tanh}$ satisfy this assumption. 

\begin{theorem}\label{thm:w1bound}
 Let $\cS^n\subset \IR^{n+1}=:\IR^{n_0}$ be the $n$-dimensional sphere, and $G\sp{\ell}, F\sp\ell:\cS^n \to \IR^{n_\ell}$, $\ell=1,\ldots, L$ be defined as above. Assume $\sigma$ is Lipschitz with constant $\Lip_\sigma$. If there is a $p>n$ and constants $B\sp\ell$, $\ell=0,\ldots,L-1$, independent of $n_1,\ldots,n_{L-1}$ such that 
 \ben{\label{eq:wellmomcond}
\IE\bcls{\bclr{W_{ij}\sp\ell}^{2p}}\leq \bbbclr{\frac{c_w\sp\ell}{n_\ell}}^p \bclr{B\sp\ell}^{p/2},
 }
 then for any $\iota>0$, there is a constant $c$ depending only on $(c_w\sp\ell,c_b\sp\ell,B\sp\ell)_{\ell=0}^L, n, p, \sigma(0), \iota$ such that 
 \bes{
 \mathsf{d}_\cW&(F\sp{L},G\sp{L}) \\
    &\leq  c\, \nr{(1+\Lip_\sigma)^{3(L-1)}}  \sum_{\ell=1}^{L-1}\bbbclr{n_{\ell+1}^{1/2}\bbbclr{\frac{n_{\ell+1}^4}{n_\ell}}^{(1-\frac{n}{p})/(6(1-\frac{n}{p}) +8(n+\iota))}\log(n_{\ell}/n_{\ell+1}^4)} \prod_{j=\ell+1}^{L-1} \IE \norm{W\sp j}_{\mathrm{op}},
 }
where $\norm{\cdot}_{\mathrm{op}}$ denotes the matrix operator norm with respect to Euclidean distance.
\end{theorem}

To the best of our knowledge, Theorem~\ref{thm:w1bound} provides the first result in the literature for bounding the law of wide random neural networks of any depth and Lipschitz activation functions, to that of a Gaussian. We emphasize in particular that the stated bounds are at the random field level, with the metric being the $W_1$ metric under the stronger $\sup$-norm topology; see Section~\ref{sec:comparison} for comparisons to prior works.
\begin{remark}\label{rem:firstrate}
\nr{To understand the bound in Theorem~\ref{thm:w1bound},
first note that in terms of the layer widths, the bound can be small only if 
\ben{\label{eq:seqlim}
n_{\ell+1}^{7+\frac{4(n+\iota)}{1-n/p}}\ll n_\ell,
}
so each layer must go to infinity polynomially faster than the next for the bound to go to zero. For the operator norm terms, if $W_{ij}\sp\ell=n_\ell^{-1/2} \widetilde W_{ij}\sp\ell$, with $\widetilde W_{ij}\sp\ell$ sub-Gaussian, then according to \cite[Exercises~4.4.6 and 4.4.7]{Vershynin2018}, we have
\be{
\IE \norm{W\sp\ell}_{\mathrm{op}} =\Theta\bclr{1 + \sqrt{n_{\ell+1}/n_\ell}}.
} 
Note that under the same assumption,~\eq{eq:wellmomcond} is satisfied for all $p\geq 2$. Thus, assuming $n_\ell$ goes to infinity fast enough relative to $n_{\ell+1}$ so that~\eq{eq:seqlim} is satisfied for large $p$, the final bound has rate 
\be{
\sum_{\ell=1}^{L-1}n_{\ell+1}^{1/2}\bbbclr{\frac{n_{\ell+1}^4}{n_\ell}}^{\frac{1}{8n+6}-\eps},
}
for any $\eps>0$. 
In the case $n=1$ and $L=2$, the rate of $n_1^{-\tfrac{1}{14}+\eps}$ matches that   given in the similar but simpler setting in \cite[Remark~1.9]{Barbour2021}, which suggests the exponent is
linked to our method, and is most certainly not optimal. (One would hope for the central limit rate of $n_1^{-1/2}$, at least  up to log factors.) 
}

\nr{
Regarding the dependence on the widths tending to infinity, it is known from \cite{Hanin2021} that asymptotic convergence to a Gaussian process holds as long as the widths all go to infinity, regardless of relative size. Therefore, our bound's dependence on the scaling of the widths is sub-optimal and is applicable in the so-called \emph{sequential limit} setting (see Section~\ref{sec:comparison} for more details). The source of this sub-optimality is an artifact of our proof.
Specifically, under a sub-Gaussian assumption on the entries of $W^{(\ell)}$, the result in Theorem~\ref{thm:w1bound} actually provides a bound for the Gaussian approximation for every layer, i.e., for the quantity $\sum_{\ell=1}^L \mathsf{d}_\cW(F\sp{\ell},G\sp{\ell})$. Indeed, for sufficiently wide neural networks, using results by~\citet[Exercise 4.4.7]{Vershynin2018} as outlined in Remark \ref{rem:firstrate}, the bound in Theorem~\ref{thm:w1bound}, with a potentially different constant $c$ having the same dependencies, applies to all layers simultaneously. For such a bound, polynomial dependence on the next layer is likely inevitable, as is the case in the classical CLT in high dimensions. 
}

\nr{
The moment bounds on the weights are used to control the modulus of continuity terms in~\eq{eq:masterresult}. The operator norm of the weights appear because we are working directly in the $\sup$-norm. The reliance of the bound on these is likely not optimal, and improvements may be achievable using our method with better technical manipulations. 
We emphasize that the question of the optimal reliance on the widths and moments of the weights is totally open, and our bound is the first and currently the only available bound at the process level that sheds light on the answer.}
\end{remark}

We next give a high-level idea behind the proof of Theorem~\ref{thm:w1bound}. Note that, conditional on the $\ell$-th layer,  layer $(\ell+1)$ is a sum of~$n_{\ell}$  random fields and a Gaussian:
\be{
F_i\sp{\ell+1}=\sum_{j=1}^{n_\ell}W\sp\ell_{i j} \sigma\bclr{F_j\sp{\ell}}  + b_i\sp{\ell},  \,\,\,\, i=1,\ldots, n_{\ell+1}.
}
Inductively, assuming an appropriate bound on the distributional distance between $F\sp{\ell}$ and $G\sp{\ell}$, we can bound the error made in the approximation
\be{
F_i\sp{\ell+1}\approx \sum_{j=1}^{n_\ell}W\sp\ell_{i j} \sigma\bclr{G_j\sp{\ell}}  + b_i\sp{\ell},  \,\,\,\, i=1,\ldots, n_{\ell+1}.
}
The field on the right-hand side has the same covariance as $G\sp{\ell+1}$, 
and hence the approximation bound in Theorem~\ref{thm:w1bound} follows by recursive application of the Stein's method approximation Theorem~\ref{thm:stnsmoothbd} (summarized in Lemma~\ref{lem:stnmethit}) to bound $\mathsf{d}_\cF$, combined with Theorem~\ref{thm:strongermetric}. 
 
As detailed in Remark \ref{rem:rate.improvement}, our next result shows that one gains an improved \nr{dependence on the widths} under the assumption that the activation function $\sigma$ is three-times differentiable, and when smoothing is only performed in the final stage, in contrast to the result of 
Theorem~\ref{thm:w1bound}, which is obtained by  smoothing at each step of the recursion.

\begin{theorem}\label{thm:rateimprovment}
Instantiate the conditions of Theorem~\ref{thm:w1bound} and assume in addition that the activation function $\sigma$ has three bounded derivatives. Then, for any $\iota>0$, there is a constant $c$ depending only on $(c_w\sp\ell,c_b\sp\ell,B\sp\ell)_{\ell=0}^L,$
$n, p, \iota,$ and $\|\sigma^{(k)}\|_\infty$, the supremum of the $k$th derivative of $\sigma$, $k=1,2,3$, such that
\be{
 \mathsf{d}_\cW(F\sp{L},G\sp{L})\leq c \sqrt{n_L}  (n_L\beta_L^2)^{(1-\frac{n}{p})/(6(1-\frac{n}{p}) +8(n+\iota))} \sqrt{\log(1/(n_L\beta_L^2))},
}
where
\begin{align}\label{eq:betaanda}
    \beta_L \coloneqq  \sum_{\ell=1}^{L-1}
    \frac{n_{\ell+1}^{3/2}}{\sqrt{n_\ell}}\prod_{j=\ell+1}^{L-1} \max\bclc{1,\IE\bcls{\norm{W\sp{j}}_{\mathrm{op}}^3}}.
\end{align}
\end{theorem}

\begin{remark}\label{rem:rate.improvement}
Under the same setting as in Remark~\ref{rem:firstrate}, with $C_L$ a constant depending only on $L$, we have
    \begin{align*}
        \beta_L^2 \leq C_L \sum_{\ell=1}^{L-1}\frac{n^3_{\ell+1}}{n_\ell} \quad \mbox{and hence} \quad \mathsf{d}_\cW(F\sp{L},G\sp{L})\leq \sqrt{n_L} \left( n_L \sum_{\ell=1}^{L-1} \frac{n^3_{\ell+1}}{n_\ell}\right)^{\frac{1}{(8n+6)} - \eps},
    \end{align*}
for any $\eps>0$, \nr{which will tend to zero as long as $n_{\ell+1}^3 \ll n_\ell$, demonstrating the improvement on the width dependence obtained by Theorem \ref{thm:rateimprovment} (specifically, in comparison to~\eqref{eq:seqlim}).}
\end{remark}

\subsubsection{Comparison to related works}\label{sec:comparison}
\cite{Eldan2021} studied Gaussian random field approximation bounds for the case of $L=2$ with Gaussian weights with three specific choices of activation functions. They used the Wasserstein-2 distance with respect to $L^p$ topology on the sphere. 
For polynomial activations they work 
with $p=\infty$, and for \textsf{ReLU} and \textsf{tanh} they work with $p<\infty$.  Following that work, \cite{Klukowski2022} derived improved bounds in the Wasserstein-2 distance with respect to $L^2$ topology, assuming the rows of the weight matrix are drawn uniformly from the sphere.
We remark that weak convergence with respect to integral norms (such as $L^p$ with $p<\infty$) does not imply weak convergence of finite dimensions, or of other natural statistics such as the maximum.   
 
\cite{Basteri2022} gives rate of convergence of finite dimensional distributions 
for general depth fully connected networks and Gaussian weights. The metric is Wasserstein-2 with respect to Euclidean norm.  The bound of \cite{Basteri2022} exhibits multivariate convergence as long as $n_\ell$ tends to infinity for each $\ell=1,\ldots,L-1$, in any order. This phenomenon is a consequence of a very good relationship between the dimension and the number of terms for the rate of convergence in the multivariate CLT, stemming from the metric used there, and the Gaussian assumptions on the weights.

\cite{bordino2023} used Stein's method to derive bounds for 
univariate distributional approximation for one-layer neural networks with Gaussian weights in the $W_1$, Kolmogorov and total variation metrics, and 
bounds for the error of the approximation of a multivariate output of the network by a Gaussian, in the $W_1$ metric. Their approach is based on a straight-forward but laborious application of a Gaussian approximation result for functions of Gaussian random variables in~\cite{vidotto2020improved}, which is a multivariate refinement of the second-order Poincar\'e inequality version of Stein's method 
introduced by~\cite{chatterjee2009fluctuations}.  

More recently, \citet[Theorem 3.13]{favaro2023quantitative}, in a paper posted to arXiv roughly two weeks after the posting of our draft, prove a rate of convergence in the $\sup$-norm under the assumptions that (i) the weights $W_{ij}$ are Gaussian, (ii) the activation function $\sigma$ is infinitely smooth with polynomially bounded derivatives, and (iii) the width of the layers all tend to infinity at the same rate. \textcolor{black}{The setting in (iii) is called as the \emph{simultaneous limit} setting in the literature, whereas the setting for our results in Theorems~\ref{thm:w1bound} and~\ref{thm:rateimprovment} is known as the \emph{sequential limit} setting; see~\cite{lee2018deep,Matthews2018, bahri2023houches} for the applications and differences between the two settings. While our bound is not informative under the \emph{simultaneous limit} setting (iii), it becomes informative under the \emph{sequential limit} setting, and applies when assuming significantly much less than (i) and (ii).} In particular, condition (ii) renders \citet[Theorem~3.13]{favaro2023quantitative} inapplicable for the widely used \textrm{ReLU} activation functions, in contrast to our results. The proof of \textcolor{black}{Theorem 3.13 of  \citet{favaro2023quantitative}} follows from the observation (also used in \cite{Basteri2022}) that due to the weights being Gaussian, the distribution of one layer conditional on the previous one is a Gaussian field, but with a conditional covariance. 
\textcolor{black}{In transport metrics with respect to Sobolev topologies, the distance between two Gaussians is determined by a distance between their covariances, which can be  controlled using results of \cite{Hanin2022}.}
While this approach leads to a remarkable result for the Gaussian weights with infinitely smooth activation functions, the method of proof does not appear to generalize beyond this specific setting, thereby providing compelling motivation for our alternative general approach.
 
\subsubsection{Future directions} Obtaining a deeper understanding of the weak convergence of wide random neural networks to Gaussian (and non-Gaussian) random fields is an active area of research. Here, we highlight a few interesting directions which can be pursued based on our work.

\textsl{Rate improvements:} There are at least two directions to explore for improving the bounds of 
Theorem~\ref{thm:w1bound} and Theorem~\ref{thm:rateimprovment}. The first, is in the case of Gaussian weights, is to understand whether the proof approach in~\cite{Basteri2022} for the multivariate setting could be extended to the random-field level. The second is to develop improved rates (in potentially weaker topologies, but still at the random field level) by combining our techniques with those in~\cite{Hanin2021}. Both directions are intriguing but appear to be non-trivial at the random field level, and we leave them as future work to investigate.

\textsl{Heavy-tailed weights}: Motivated by constructing priors for Bayesian inference for neural networks, ~\citet[Section 2.2]{Neal1996} also heuristically examined the limits of single layer neural networks with the entries of the weight matrices being stable random variables. Recently, several works~\citep{der2005beyond,jung2021alpha,Bordino2022,lee2022deep,fortuin2022bayesian,favaro2023deep,Bordino2022} showed that the limits of such neural networks (including deep ones) converge weakly to appropriately defined stable random fields. An interesting question that arises is whether one can establish quantitative distributional approximation bounds in the heavy-tailed  setting. Our work provides a step in this direction. Indeed, our main result in Theorem~\ref{thm:strongermetric} is immediately applicable. The remaining challenge will be in establishing a version of Theorem~\ref{thm:stnsmoothbd} for stable random fields. This could potentially be accomplished by extending recent works, for example,~\cite{xu2019stable,arras2022somea,arras2022someb,chen2023multivariate}, on multivariate stable approximations to the random field setting.

The rest of the article is organized as follows.  Section~\ref{sec:smth} defines and develops properties of our smoothing Gaussian process, which are then used in Section~\ref{sec:wass} to prove our general smoothing result, Theorem~\ref{thm:strongermetric}. Section~\ref{sec:firstterm} develops Stein's method for Gaussian processes, culminating in Theorem~\ref{thm:stnsmoothbd}, which is used to bound $\mathsf{d}_\cF$. Section~\ref{sec:2and3terms} uses classical quantitative chaining arguments along with heat kernel bounds to prove Lemma~\ref{lem:momepsreg}, which gives an easily applied method for bounding $\IE\norm{F-F_\eps}_\infty$. Finally, Section~\ref{sec:widednn} uses the theory developed in the previous sections to prove our wide neural network approximation results, Theorems~\ref{thm:w1bound} and~\ref{thm:rateimprovment}.

\vspace{0.1in}
\noindent \textbf{Acknowledgments.} We thank Volker Schlue for discussions regarding certain differential geometric aspects and Max Fathi for the suggestion to look at the Laplacian for smoothing. This project originated at the ``Stein's Method: The Golden Anniversary'' workshop organized by the Institute for Mathematical Sciences at the National University of Singapore in June--July, 2022. We thank the institute for the hospitality and the organizers for putting together the stimulating workshop. KB was supported in part by National Science Foundation (NSF) grant DMS-2053918.

\section{Gaussian smoothing for random fields indexed by the sphere}\label{sec:smth}
We begin by constructing our Gaussian smoothing random field, with its covariance defined based on the powers of Laplacian operators, and specifying its Cameron-Martin space. 

\subsection{Constructing a covariance and Cameron-Martin space from the Laplacian}\label{sec:constructcov}
To define our smoothing Gaussian random field, we construct a covariance function based on the Laplacian on the $n$-sphere $\cS^n$, which we view as embedded in $\IR^{n+1}$,
\[
\cS^n=\clc{x\in \IR^{n+1}: \norm{x}_2=1}.
\]
A standard way to define the Laplacian on the sphere is to ``lift'' functions $f:\cS^n \to \IR$ to $\tilde f:\IR^{n+1}\setminus \{0\} \to \IR$ by
\be{
\tilde f(x) = f\clr{x/\norm{x}_2}.
}
Letting $\widetilde \Delta$ denote the usual Laplacian on $\IR^{n+1}$, we can then define the 
Laplacian $\Delta$ acting on twice differential functions on $\mathcal{S}^n$ by 
\be{
\Delta f (x)= \widetilde \Delta \tilde f(x), \,\,x\in\cS^n;
}
see, for example, \cite[Corollary~1.4.3]{dai2013approximation}. The negative of the Laplacian $(-\Delta)$ is a positive definite operator on $L^2(\cS^n; \IR)$ and has an orthonormal basis given by  \emph{spherical harmonics}. 
The eigenvalues of $(-\Delta)$ are 
\ben{\label{eq:lapeval}
\lambda_k=k(k+n-1),  \,\,\, k=0,1,2,\ldots,
}
 and an orthonormal basis for the eigenspace associated to $\lambda_k$ is given by a collection of polynomials $\sH_k=\bclc{\phi_k\sp{1},\ldots, \phi_k\sp{d_k}}$, with 
\ben{\label{eq:dkbd}
d_k\coloneqq\mathrm{dim} \, \sH_k = \frac{2k+n-1}{k} \binom{n+k-2}{k-1};
}
see, for example, \cite[Corollary 1.1.4]{dai2013approximation}. The union $\bigcup_{k \geq 0}\sH_k$, of the sets of all basis vectors for the $k^{th}$ eigenspace,
gives an orthonormal basis for $L^2(\cS^n; \IR)$. From here, we define the \emph{zonal harmonics}
\ben{\label{eq:zkmerc}
Z_k(x,y)= \sum_{j=1}^{d_k} \phi_k\sp{j}(x) \phi_k\sp{j}(y), 
}
that for $n\geq2$ satisfy 
\ben{\label{eq:zkdef}
Z_k(x,y) = \frac{\Gamma\bclr{(n+1)/2}(2k+n-1) }{2  \pi^{(n+1)/2}(n-1)} C_k^{(n-1)/2} \bclr{\angle{x,y}},
}
where $C_k^\lambda, \lambda>0, k \geq -1$ are the \emph{Gegenbauer polynomials} defined by the three term recurrence, for $x\in [-1,1]$,
\be{
C_{k+1}^\lambda(x) = \frac{2(k+\lambda)}{k+1} x C_k^\lambda(x) - \frac{k+2\lambda -1}{k+1} C_{k-1}^\lambda(x) \quad \mbox{for $k \geq 1$,}
} 
with initial values $C_{-1}^\lambda \equiv 0$ and $C_{0}^\lambda \equiv 1$. For $n=1$, 
$Z_k(x,y)=\pi^{-1} \cos\bclr{k (\theta_x-\theta_y)}$, where $\theta_x,\theta_y$ are the polar angles of $x,y$, i.e., $x=(\cos(\theta_x),\sin(\theta_x))$. For our purposes, the key property of $Z_k$ is that
\ben{\label{eq:zkbds}
\abs{Z_k(x,y)} \leq Z_k(x,x) = \frac{\Gamma\bclr{(n+1)/2}}{2  \pi^{(n+1)/2}} d_k, 
}
see \cite[Corollary 1.2.7]{dai2013approximation}, noting their different normalization at (1.1.1) of the inner product on the sphere.
Thus, for any $\iota>0$, we can define the kernel $C\sp\iota=(C\sp\iota_{ij})_{i,j=1}^d$ on~$\cS^n$ by
\ben{\label{eq:smthcov}
C\sp\iota_{ij}(x,y)=\delta_{ij}\sum_{k\geq 1} \frac{ Z_k(x,y)}{\lambda_k^{n_\iota}},
}
where $n_\iota:=(n+\iota)/2$. Because\footnote{For two functions $f,g$, by $f\asymp g$ means that there exists absolute constants $c,C >0$ such that $c|g| \leq |f| \leq C |g|$.} $\lambda_k \asymp k^2$ and $d_k \asymp k^{n-1}$, by \eqref{eq:zkbds} we see that 
$\abs{Z_k(x,y)}$ is $\bigo(k^{n-1})$ uniformly, hence the sum \eqref{eq:smthcov}  converges absolutely and uniformly. Since each $Z_k$ is continuous and the sphere is compact, $C\sp\iota$ is continuous and positive definite due to  the decomposition~\eq{eq:zkmerc}. 
We fix $n\geq 2$ and $\iota>0$, and set $C=C\sp\iota$ for the remaining part of this section. With our covariance kernel in hand, we define our smoothing random field $S$ and
its Cameron-Martin space. 

\begin{definition}[The Smoothing Gaussian random field $S$ and its Cameron-Martin space]
Let $S$ be the centered $\IR^d$-valued Gaussian random field indexed by $\cS^{n}$ with covariance function $C$ given by~\eq{eq:smthcov}. 
Let $\mathbf{e}_i\in \IR^d$, for $i=1\ldots, d$ be the standard basis vectors for $\IR^d$. The associated orthonormal decomposition for an  $h\in L^2(\cS^n; \IR^d)$ is given by
\ben{\label{eq:l2exp}
h= \sum_{i=1}^d \mathbf{e}_i \sum_{k\geq 1}\sum_{j=1}^{d_k}  h_{k,i}\sp{j} \phi_k\sp{j} \quad \mbox{where} \quad h_{k,i}^{(j)}=\int_{\mathcal{S}^n} h_i(x)\phi_k^{(j)}(x) \mathrm{d}x,
}
where $\mathrm{d}x$ is the volume measure on the sphere.
We define the \emph{Cameron-Martin} or \emph{Reproducing Kernel Hilbert} space $H$ of $S$ to be the subset of $L^2(\cS^n; \IR^d)$  defined by
\ben{\label{eq:cmspdef}
H=\bbbclc{h\in L^2(\cS^n; \IR^d) :  \sum_{k\geq 1}\lambda_k^{n_\iota}\sum_{i=1}^d\sum_{j=1}^{d_k}(h_{k,i}\sp{j})^2<\infty},
}
equipped with
inner product
\be{
\angle{h, g}_H:=\sum_{k\geq 1}\lambda_k^{n_\iota}\sum_{i=1}^d\sum_{j=1}^{d_k}h_{k,i}\sp{j}g_{k,i}\sp{j}.
}
\end{definition}

There is the following alternative description of the Cameron-Martin space and inner product. 
We define the fractional Laplacian operator $(-\Delta)^\alpha$ for any $\alpha>0$ through the 
orthonormal basis $(-\Delta)^\alpha\phi_k\sp j:=\lambda_k^\alpha \phi_k\sp j$, and 
for $h:\cS^n\to \IR^d$
we write $(-\Delta)^\alpha h$ for the fractional-Laplacian applied coordinate-wise.
\begin{proposition}\label{prop:cmaltdes}
If $h, g\in L^2(\cS^n;\IR^d)$ are such that $(-\Delta) ^{\frac{1}{2}n_\iota}h, (-\Delta)^{\frac{1}{2}n_\iota} g \in L^2(\cS^n;\IR^d)$,
then $h,g \in H$ and
\be{
\angle{h,g}_H = \bangle{(-\Delta)^{\frac{1}{2}n_\iota} h, (-\Delta)^{\frac{1}{2}n_\iota} g }_{\scaleto{L^2(\cS^n;\IR^d)}{7.5pt}}.
}
\end{proposition}
\begin{proof}
First note that
\be{
\bangle{(-\Delta)^{\frac{1}{2}n_\iota} h, (-\Delta)^{\frac{1}{2}n_\iota} g }_{\scaleto{L^2(\cS^n;\IR^d)}{7.5pt}} = \sum_{i=1}^d \int_{\cS^{n}} (-\Delta)^{\frac{1}{2}n_\iota} h_i(x) (-\Delta)^{\frac{1}{2}n_\iota} g_i(x) \, \mathrm{d}x.
}
Thus, by additivity, it suffices to show the result for $d=1$. Since $(-\Delta)^{\frac{1}{2}n_\iota}h\in L^2(\cS^n)$, we can compute the coefficients
in its $L^2(\cS^n)$ expansion~\eq{eq:l2exp} as
\ba{
\int (-\Delta)^{\frac{1}{2}n_\iota}h(x) \phi_k\sp{j}(x) \mathrm{d}x 
	&=\sum_{\ell\geq 1} \lambda_\ell^{\frac{1}{2} n_\iota}\sum_{i=1}^{d_k}h_{\ell}\sp{i} \int \phi_\ell\sp{i}(x)\phi_k\sp{j}(x) \mathrm{d}x \\
	& = \lambda_k^{\frac{1}{2}n_\iota} h_k\sp{j},
}
where the second equality follows from orthonormality. 
Thus, we have that 
\be{
\bangle{(-\Delta)^{\frac{1}{2}n_\iota} h, (-\Delta)^{\frac{1}{2}n_\iota} g }_{\scaleto{L^2(\cS^n;\IR^d)}{7.5pt}} = \sum_{k\geq1} \lambda_k^{n_\iota} \sum_{j=1}^{d_k} h_k\sp{j} g_k\sp{j} = \angle{h, g}_H. 
}
\end{proof}

To explicitly state the Cameron-Martin change of measure formula for $S$, we first provide its 
Karhunen-Loeve expansion; see \cite[Chapter~3]{adler2007random}.

\begin{theorem}[Karhunen-Loeve Expansion for the Smoothing Gaussian random field $S$]\label{thm:KL}
There exists $(\cZ_{k,i}\sp{j}: k\geq1, 1\leq j\leq d_k, 1\leq i \leq d)$ independent centered normal random variables with $\Var(\cZ_{k,i}\sp{j}) = \lambda_k^{-n_\iota}$ such that
\be{
S_i=\sum_{k\geq1} \sum_{j=1}^{d_k} \cZ_{k,i}\sp{j} \phi_{k}\sp{j},
}
where the convergence holds in $L^2$ and almost surely, uniformly on $\cS^n$.
\end{theorem}

With this result, we have the following natural definition. 
\begin{definition}[Paley-Wiener integral]
For $h\in H$ with $L^2$ expansion~\eqref{eq:l2exp}, the Paley-Wiener integral  with respect to $S$ is the centered normal random variable with variance $\angle{h,h}_H$ given by
\be{
\angle{h, S}_H :=\sum_{i=1}^d \sum_{k\geq1} \lambda_k^{n_\iota} \sum_{j=1}^{d_k} \cZ_{k,i}\sp{j} h_{k,i}\sp{j},
}
where the $\cZ_{k,i}\sp{j}$ are as in Theorem~\ref{thm:KL}.
\end{definition}

We can now 
formally state the Cameron-Martin change of measure formula for $S$, which follows from an application of a theorem of \cite{kakutani1948equivalence}
for absolute continuity of infinite product measure.

\begin{theorem}[Cameron-Martin change of measure for the Smoothing Gaussian random field $S$]\label{thm:cm}
For any $h\in H$, $\law(h+S)$ has Radon-Nikodym derivative with respect to $\law(S)$, given by
\be{
\frac{d\law(h + S)}{d\law(S)} = \exp\bclc{\angle{h,S}_H - \tfrac{1}{2} \angle{h,h}_H}.
}
\end{theorem}

\subsection{Regularization to the Cameron-Martin space}
In the previous section, we provided the Cameron-Martin change of measure 
formula for our smoothing Gaussian random field $S$, but it only applies to functions $f\in \rC(\cS^n;\IR^d)$ that are \nr{in the Cameron-Martin space $H$ defined at~\eqref{eq:cmspdef}, or, according to  Proposition~\ref{prop:cmaltdes}, that are} sufficiently smooth. Thus, we define the $\eps$-regularization of $f$ by 
\ben{\label{eq:genrep}
f_\eps(x) = \bclr{f_{\eps,i}(x)}_{i=1}^d = \bclr{e^{\eps \frac{\Delta}{2}} f_i(x)}_{i=1}^d = \sum_{i=1}^d \mathbf{e}_i \sum_{k\geq 1} e^{-\frac{\eps \lambda_k}{2}} \sum_{j=1}^{d_k} f_{k,i}\sp{j} \phi_{k}\sp{j}(x). 
}
The $\eps$-regularized $f_\eps(x)$ equals $\IE\bcls{f(B\sp{x}_\eps)}$, where $(B_t\sp{x})_{t\geq0}$ is a $d$-dimensional Brownian motion run on the sphere started from $x$; see \cite{bakry2014analysis}.
The next proposition uses this representation of 
 $f_\eps$ in terms of the ``heat kernel'' for Brownian motion, which will be
useful to derive smoothness properties.
\begin{proposition}\label{prop:htker}
Let 
\ben{\label{eq:htker}
p(x,y;\eps) = \sum_{k\geq1} e^{-\frac{\eps \lambda_k}{2}} \sum_{j=1}^{d_k} \phi_k\sp{j}(x) \phi_k\sp{j}(y)= \sum_{k\geq1} e^{-\frac{\eps \lambda_k}{2}}Z_k(x,y),
}
using the definition of $Z_k$ in~\eq{eq:zkmerc}. Then for any bounded and measurable $f:\cS^n\to\IR^d$, 
\ben{\label{eq:alternative}
f_{\eps,i}(x) = \int_{\cS^n} p(x,y; \eps) f_i(y) \mathrm{d}y.
}
\end{proposition}
\begin{proof}
Dropping the subscript $i$, we have by~\eq{eq:htker},\eq{eq:zkbds},~\eq{eq:dkbd}, and Fubini's theorem that
\ba{
\int_{\cS^n} p(x,y; \eps) f(y) \mathrm{d}y 
	& = \sum_{k\geq1} e^{-\frac{\eps \lambda_{k}}{2}} \int_{\cS^n}f(y) Z_{k}(x,y) \mathrm{d}y  \\
	&= \sum_{k\geq1}  e^{-\frac{\eps \lambda_{k}}{2}}\sum_{j=1}^{d_k}f_{k}\sp{j} \phi_{k}\sp{j}(x),
}
which is the same as~\eq{eq:genrep}.
\end{proof}

We are now in position to derive bounds on $\norm{(-\Delta)^\alpha f_\eps}$.
\begin{proposition}\label{prop:lapmbd}
If $f:\cS^n\to \IR^d$ is bounded and measurable, then for any $\alpha>0$, $(-\Delta)^\alpha f_\eps$ exists and  $(-\Delta)^\alpha f_\eps\in L^2(\cS^n;\IR^d)$.
Moreover, there is a constant $c=c(n,\alpha)$ depending only on $n$ and $\alpha$ such that
\be{
\norm{(-\Delta)^\alpha f_{\eps,i}}_\infty \leq c\,   \norm{f_i}_\infty \eps^{-(2\alpha+n)/2}. 
}
\end{proposition}
\begin{proof}
By~\eq{eq:zkdef} each (lifted) $Z_k(\cdot,y)$ is infinitely differentiable, with derivatives growing in absolute value at most polynomially in $k$. Thus, using~\eq{eq:htker}, that $\lambda_k\asymp k^2$, and dominated convergence, $(-\Delta_x)^\alpha p(x,y;\eps)$ is well-defined 
and 
\be{
(-\Delta_x)^{\alpha} p(x,y; \eps)=\sum_{k\geq1} e^{-\frac{\eps \lambda_k}{2}}(-\Delta_x)^\alpha Z_k(x,y).
}

Now, dropping the $i$ subscript, using~\eq{eq:htker} and~\eq{eq:zkmerc}, followed by~\eq{eq:zkbds},~\eq{eq:dkbd} to find $d_k=\bigo(k^{n-1})$ and \eq{eq:lapeval}, so as to apply dominated convergence, we have 
\ba{
\babs{(-\Delta)^{\alpha} f_{\eps}(x) }
	& =\bbbabs{ \int_{\cS^n} (-\Delta_x)^{\alpha} p(x,y; \eps) f(y) \mathrm{d}y } \\
	&\leq \norm{f}_\infty \int_{\cS^n} \babs{ (-\Delta_x)^{\alpha} p(x,y; \eps) } \mathrm{d}y \\
	&= \norm{f}_\infty \int_{\cS^n} \bbbabs{\sum_{k\geq1} e^{-\frac{\eps \lambda_k}{2}}(-\Delta_x)^\alpha Z_k(x,y) } \mathrm{d}y \\
	&= \norm{f}_\infty \int_{\cS^n} \bbbabs{\sum_{k\geq1} \lambda_k^\alpha e^{-\frac{\eps \lambda_k}{2}}Z_k(x,y) } \mathrm{d}y \\
	&\leq \norm{f}_\infty  \sum_{k\geq1} \lambda_k^\alpha e^{-\frac{\eps \lambda_k}{2}} d_k  \\
	&\leq c \norm{f}_\infty \sum_{k\geq1} k^{2\alpha+n-1} e^{-\frac{\eps k^2}{2}}.
	} 
By comparing this sum with 
\be{
\int_0^\infty (\eps/2)^{(2\alpha+n)/2} x^{2\alpha+n-1} e^{-\eps x^2/2} \mathrm{d}x = \tfrac{1}{2}\Gamma(\alpha+n/2),
}
we find
\be{
\babs{(-\Delta)^{\alpha} f_{\eps}(x) }\leq c \,  \eps^{-(2\alpha+n)/2} \norm{f}_\infty,
}
where $c$ is a constant depending only on $n$ and $\alpha$, as desired.
\end{proof}

Propositions~\ref{prop:cmaltdes} and~\ref{prop:lapmbd} imply that $f_\eps\in H$ for bounded and measurable $f$, and also give the following lemma bounding $\abs{\angle{f_\eps, g_\eps}_H}$, whose proof is straightforward.
\begin{lemma}\label{lem:inprodbd}
If $f,g$ are bounded and measurable functions $\cS^n\to\IR^d$, then there is a constant $c=c\clr{n,\iota}$ depending
only on $n$ and $\iota$ such that 
\be{
\babs{\angle{f_\eps, g_\eps}_H} =\babs{\bangle{ (- \Delta)^{\frac{1}{2} n_\iota} f_\eps, (-\Delta)^{\frac{1}{2}n_\iota}  g_\eps}_{\scaleto{L^2(\cS^n;\IR^d)}{7.5pt}}} \leq c \, d\norm{f}_\infty \norm{g}_\infty \eps^{-(2n + \iota)}.
}
\end{lemma}

\subsection{Smoothing using $S$ and regularization}
We now use the $f_\varepsilon$ regularization given in the last section to define a $(\eps,\delta)$-regularized version of a test function  $\zeta$. The following result is an analog of \cite[Lemma~1.10]{Barbour2021}.
\begin{theorem}\label{thm:regzet}
Let $\zeta:\rC(\cS^n;\IR^d)\to \IR$ and, for $f:\cS^n\to\IR^d$ bounded and measurable, define
\be{
\zeta_{\eps,\delta}(f):= \IE[\zeta(f_\eps+ \delta S)],
}
where $f_\eps$ is the $\eps$-regularization defined at~\eq{eq:genrep}.  
If $\zeta$ is bounded or Lipschitz, 
then 
 $\zeta_{\eps,\delta}$ is infinitely differentiable. Moreover, for every $k \geq 0$ there is a constant $c$ depending only on $k$, $n$ and $\iota$, such that if $\zeta$ is bounded, then 
\be{
\norm{D^k \zeta_{\eps,\delta}}\leq c \, d^{k/2} \delta^{-k} \eps^{- k(n+\iota)} \norm{\zeta}_\infty,
}
and if  $\zeta$ is $1$-Lipschitz and $h:\mathcal{S}^n \rightarrow \mathbb{R}^d$ is bounded and measurable,
then
\ben{\label{eq:regzetderiv}
\norm{D^{k} \zeta_{\eps,\delta}(f)-D^{k} \zeta_{\eps,\delta}(h)}\leq c \,d^{k/2}  \delta^{-k} \eps^{-k(n+\iota)}\norm{f-h}_\infty.
}

\end{theorem}
\begin{proof}
The proof is closely related to that of \cite[Lemma~1.10]{Barbour2021}, where
the Cameron-Martin inner product and $\eps$-regularization are simpler. Intuition behind
the manipulations below can be found there.

Firstly, $\zeta_{\eps,\delta}$ is clearly well-defined if $\zeta$ is bounded. If $\zeta$ is $C$-Lipschitz, then
\be{
\babs{\zeta_{\eps,\delta}(f)-\zeta(f_\eps)} \leq C \delta \IE\norm{S}_\infty <\infty,
}
where the last inequality is Fernique's theorem~\citep{fernique1970integrabilite}. Moreover, $\zeta_{\eps,\delta}$
is measurable since, from~\eq{eq:alternative}, $f\mapsto f_\eps$ is continuous with respect to $\sup$-norm, as is $(f,g)\mapsto f+ g$ in product topology.
Thus, $(f,s) \mapsto \zeta(f_\eps + \delta s)$ is measurable with respect to product topology.

We claim that for $\zeta$ bounded, $k\geq1$ and $g\sp{i}\in \rC(\cS^n;\IR^d)$, $i=1,\ldots, k$, we have
\ben{\label{eq:dkexp1}
D^{k} \zeta_{\eps, \delta}(f) [g\sp{1}, \ldots ,g\sp{k}]  
= \IE\bbbcls{  \zeta (\delta S) e^{\Psieps(f)}  \sum_{ \pi\in\cP_{k,2}}  
            \prod_{b\in \pi}  D^{| b |} \Psieps(f) [g\sp{b}]  },
}
where
\be{
\Psieps(f)= \tfrac{1}{\delta}\angle{f_\eps, S}_H-\tfrac{1}{2\delta^{2}}\angle{f_\eps, f_\eps}_H.
}
In \eqref{eq:dkexp1} $\cP_{k,2}$  is the set of all partitions of $\{1,\ldots,k\}$, whose blocks have at most $2$ elements; 
  $b\in\pi$ means that $b$ is a block of $\pi$,  and we denote its cardinality by $|b|$. When $b=\{i\}$ the expression $D^{| b |} \Psieps(f) [g\sp{b}] $ is defined as
  \be{
 D^{| b |} \Psieps(f) [g\sp{b}]= D \Psieps(f) [g\sp{i}] =\delta^{-1} \bangle{g_\eps\sp{i}, S- \delta^{-1} f_\eps}_H ,
   }
   and when $|b|=2$ is given by 
   $b=\{i_1,i_2\}$, then 
     \be{
 D^{| b |} \Psieps(f) [g\sp{b}]= D^{2} \Psieps(f) [g\sp{i_1}, g\sp{i_2}] = - \delta^{-2} \angle{g_\eps\sp{i_1}, g_\eps\sp{i_2}}_H,
   }   
which we note does not depend on $f$. Compare to \cite[Equation (2.11)]{Barbour2021} with $\Theta\equiv0$.
Assuming~\eq{eq:dkexp1}, the Cameron-Martin Theorem~\ref{thm:cm} implies  that 
\besn{\label{eq:dkexp2}
D^{k} \zeta_{\eps, \delta}(f) [g\sp{1}, \ldots ,g\sp{k}]  
= \IE\bbbcls{  \zeta (f_\eps +\delta S)   \sum_{ \pi\in\cP_{k,2}}  
            \prod_{b\in \pi}  \widehat D^{| b |} \Psieps(f) [g\sp{b}]  },
}
where  $\widehat D^2 \Psieps(f) = D^2 \Psieps(f)$, and 
\be{
\widehat D \Psieps(f)[g]= \delta^{-1} \bangle{g_\eps\sp{i}, S}_H\sim \mathrm{Normal}\bclr{0, \delta^{-2} \angle{g_\eps\sp{i}, g_\eps\sp{i}}_H},
}
and we note that $\widehat D^{|b|} \Psieps(f)$ does not depend on $f$ for $|b| \in \{1,2\}$. Technically, we are applying the Cameron-Martin change of measure formula to the joint distribution of the random variables $\bclr{\bangle{g_\eps\sp{i}, S- \delta^{-1} f_\eps}_H}_{i=1}^k$ and $(S-\delta^{-1} f_\eps)$, which follows in a straightforward way from Kakutani's theorem and the definition of the Paley-Wiener integral.

By Lemma~\ref{lem:inprodbd}, we have
\ben{\label{eq:ghbd}
\babs{\angle{g_\eps\sp{i_1}, g_\eps\sp{i_2}}_H}  \leq c\, d \norm{g\sp{i_1}}_\infty \norm{g\sp{i_2}}_\infty \eps^{-(2n+\iota)},
}
where $c$ is a constant depending only on $n_\iota$.
Thus, if $\zeta$ is bounded, we have 
\be{
\babs{D^{k} \zeta_{\eps, \delta}(f) [g\sp{1}, \ldots ,g\sp{k}]}
\leq \norm{\zeta}_\infty \sum_{\pi \in \cP_{k,2}}\mathop{\prod_{b\in \pi}}_{\abs{b}=2}\babs{\widehat D^{2} \Psieps(f) [g\sp{b}]}\IE \mathop{\prod_{b\in \pi}}_{\abs{b}=1} \babs{\widehat D \Psieps(f) [g\sp{b}]},  
}
and then the definition of $\widehat D^k \Psieps$,~\eqref{eq:ghbd}, and H\"older's inequality imply
\be{
\babs{D^{k} \zeta_{\eps, \delta}(f) [g\sp{1}, \ldots ,g\sp{k}]} \leq c \, d^{k/2} \delta^{-k} \eps^{- k(n +\iota)} \norm{\zeta}_\infty \prod_{i=1}^k \norm{g\sp{i}}_\infty, 
}
where $c$ depends on $k$ (through the sum over $\cP_{k,2}$ and the absolute moments up to order $k$ of standard normal variables) and $n_\iota$, as desired. 

Assume now  $\zeta$ is  $1$-Lipshitz, and letting  $f,h\in \rC(\cS^n;\IR^d)$ and recalling that $\widehat D^{\abs{b}} \Psieps(f)[g\sp{b}]=\widehat D^{\abs{b}} \Psieps(h)[g\sp{b}]$, 
\eq{eq:dkexp2} implies
\bes{
D^{k} \zeta_{\eps,\delta}(f)[g\sp{1}, \ldots ,g\sp{k}] -D^{k}& \zeta_{\eps,\delta}(h)[g\sp{1}, \ldots ,g\sp{k}] \\ 
 &= \IE\bbbcls{ \bclr{ \zeta (f_\eps +\delta S) -\zeta (h_\eps +\delta S)} \sum_{ \pi\in\cP_{k,2}}  
            \prod_{b\in \pi}  \widehat D^{| b |} \Psieps(f) [g\sp{b}]  },
}
and using that $\zeta$ is Lipschitz and~\eqref{eq:alternative}, we have 
\be{
\babs{\zeta (f_\eps +\delta S) -\zeta (h_\eps +\delta S)} \leq \norm{f_\eps-h_\eps}_\infty \leq \norm{f-h}_\infty.
}
With this,~\eqref{eq:regzetderiv} follows in exactly the same way as the bounded case.

To establish~\eq{eq:dkexp1}, we use induction.  For $k=1$, the Cameron-Martin Theorem~\ref{thm:cm} implies 
\be{
\zeta_{\eps, \delta}(f+g)- \zeta_{\eps, \delta}(f) = \IE\bcls{\zeta(\delta S) \clr{e^{\Psieps(f+g)}- e^{\Psieps(f)}}},
}
so by the bounded- or Lipschitz-ness of $\zeta$ and the Cauchy-Schwarz inequality, it  is enough to show that 
\ben{\label{eq:normglito}
\IE\bbcls{\bclr{e^{\Psieps(f+g)-\Psieps(f)}- 1-D\Psieps(f)[g]}^2} = \lito\bclr{\norm{g}_\infty^2}.
}
But
\be{
\Psieps(f+g)-\Psieps(f)=D\Psieps(f)[g] - \frac{1}{2\delta^2}\angle{g_\eps, g_\eps}_H,
}
with $D\Psieps(f)[g]\sim \mathrm{Normal}(-\delta^{-2} \angle{f_\eps,g_\eps}_H,\delta^{-2} \angle{g_\eps, g_\eps}_H)$,
and so straightforward computing shows
\ba{
&\IE\bbcls{\bclr{e^{\Psieps(f+g)-\Psieps(f)}- 1-D\Psieps(f)[g]}^2} \\
	 =~&e^{\delta^{-2} \angle{g_\eps, g_\eps}_H-2\delta^{-2} \angle{f_\eps,g_\eps}_H} + \bbbclr{1-\frac{\angle{f_\eps,g_\eps}_H}{\delta^2} }^2 + \frac{ \angle{g_\eps,g_\eps}_H}{\delta^2} \\
	\quad~&-2 e^{-\delta^{-2} \angle{f_\eps,g_\eps}_H}\bbbclr{1-\frac{ \angle{f_\eps,g_\eps}_H}{\delta^2}+ \frac{ \angle{g_\eps,g_\eps}_H}{\delta^2}}, 
}
which, using  Lemma~\ref{lem:inprodbd},  is easily seen to be  $\lito\bclr{\norm{g}_\infty^2}$, 
as desired. Compare to \cite[(2.12-15)]{Barbour2021}.

Assuming~\eq{eq:dkexp1} holds for $k$, we want to show it holds for $k+1$. We write
\ban{
D^{k} \zeta_{\eps, \delta}(f+g) &[g\sp{1}, \ldots ,g\sp{k}]  - D^{k} \zeta_{\eps, \delta}(f) [g\sp{1}, \ldots ,g\sp{k}]  \notag   \\
	&=\IE\bbbcls{  \zeta (\delta S) \bclr{e^{\Psieps(f+g)}-e^{\Psieps(f)}}  \sum_{ \pi\in\cP_{k,2}}  
            \prod_{b\in \pi}  D^{| b |} \Psieps(f) [g\sp{b}]  } \label{eq:t1pr}  \\ 
           &\qquad + \IE\bbbcls{  \zeta (\delta S) e^{\Psieps(f)}  \sum_{ \pi\in\cP_{k,2}}  
            \bbclr{\prod_{b\in \pi}  D^{| b |} \Psieps(f+g) [g\sp{b}] -\prod_{b\in \pi}  D^{| b |} \Psieps(f) [g\sp{b}] }} \label{eq:t2pr}
}
Because of~\eq{eq:normglito}, the term~\eq{eq:t1pr} is equal to
\ben{\label{eq:t1fe}
\IE\bbbcls{  \zeta (\delta S) D\Psieps(f)[g] \sum_{ \pi\in\cP_{k,2}}  
            \prod_{b\in \pi}  D^{| b |} \Psieps(f) [g\sp{b}]  } + \lito\bclr{\norm{g}_\infty}.
}
Now working on~\eq{eq:t2pr}, noting that $D^2\Psieps(f+g)=D^2 \Psieps(f)$ and $D\Psieps(f+g)[h]=D\Psieps(f)[h]+D^2\Psieps(f)[h,g]$, we find 
\ba{
\sum_{ \pi\in\cP_{k,2}}  
            \bbclr{\prod_{b\in \pi}& D^{| b |} \Psieps(f+g) [g\sp{b}] -\prod_{b\in \pi}  D^{| b |} \Psieps(f) [g\sp{b}] } \\
            &=
\sum_{ \pi\in\cP_{k,2}}  
            \mathop{\prod_{b\in \pi}}_{\abs{b}=2}  D^{2} \Psieps(f) [g\sp{b}] \bbbclc{\mathop{\prod_{b\in \pi}}_{\abs{b}=1} \bbclr{D\Psieps(f) [g\sp{b}] + D^2\Psieps(f)[g\sp{b},g]}-\mathop{\prod_{b\in \pi}}_{\abs{b}=1} D\Psieps(f) [g\sp{b}]} \\ 
            &= \sum_{ \pi\in\cP_{k,2}}  
            \mathop{\prod_{b\in \pi}}_{\abs{b}=2}  D^{2} \Psieps(f) [g\sp{b}] \bbbclc{\mathop{\sum_{b\in \pi}}_{\abs{b}=1} D^2\Psieps(f)[g\sp{b},g] \mathop{\prod_{b\not=a \in \pi}}_{\abs{a}=1} D\Psieps(f)[g\sp{a}]} +\mathbf{o}\bclr{\norm{g}_\infty},
}
where $\mathbf{o}\bclr{\norm{g}_\infty}$ is a random variable, say $X=X(g)$ depending on $g$, such that $\IE\bcls{\abs{X}^p}^{1/p}=\lito\bclr{\norm{g}_\infty}$ for any $p\geq 2$. This is because $D\Psieps(f)[g\sp{b}]$ is Gaussian, and $D^2\Psieps(f)[g\sp{b}, g]=\bigo\bclr{\norm{g}_\infty}$, by Lemma~\ref{lem:inprodbd}. 
Thus, up to a $\lito\bclr{\norm{g}_\infty}$ term,~\eq{eq:t2pr} is equal to
\ben{\label{eq:t2fe}
\IE\bbbcls{  \zeta (\delta S) e^{\Psieps(f)}  \sum_{ \pi\in\cP_{k,2}}  
            \mathop{\prod_{b\in \pi}}_{\abs{b}=2}  D^{2} \Psieps(f) [g\sp{b}] \bbbclc{\mathop{\sum_{b\in \pi}}_{\abs{b}=1} D^2\Psieps(f)[g\sp{b},g] \mathop{\prod_{b\not=a \in \pi}}_{\abs{a}=1} D\Psieps(f)[g\sp{a}]} }.
            }
Combining~\eq{eq:t1fe} and~\eq{eq:t2fe} completes the induction.
\end{proof}

\section{Proof of the Wasserstein bound}\label{sec:wass}

Armed with Theorem~\ref{thm:regzet}, we follow the strategy described in Section~\ref{sec:master} to prove our master theorem.

\begin{proof}[\textbf{Proof of Theorem~\ref{thm:strongermetric}}]
To achieve a bound in the Wasserstein distance, let $\zeta:\rC(\cS^n;\IR^d)$ be a Lipschitz function and let~$\zeta_{\eps,\delta}$ be defined as in Theorem~\ref{thm:regzet}. The triangle inequality yields
\ban{
\abs{\IE& \zeta(F)-\IE \zeta(H)} \notag \\
&\leq \babs{\IE\cls{\zeta_{\eps,\delta}(F)}-\IE\cls{\zeta_{\eps,\delta}(H)}} + \babs{\cls{\IE \zeta(F)}-\IE\cls{\zeta_{\eps,\delta}(F)}}  + \babs{\cls{\IE\cls{\zeta(H)}-\IE\zeta_{\eps,\delta}(H)}}. \label{eq:dwtin1}
}

For the first term of~\eq{eq:dwtin1}, we use~\eqref{eq:regzetderiv} of Theorem~\ref{thm:regzet} with $k=2$ and the definition \eqref{def:calF} of $\cF$ to find
\be{
\babs{\IE\cls{\zeta_{\eps,\delta}(F)}-\IE\cls{\zeta_{\eps,\delta}(H)}} \leq c \, d \, \delta^{-2} \eps^{-  2(n+\iota)} \mathsf{d}_{\cF}(F,H).
}

For the second term, using the definition of $\zeta_{\eps,\delta}$ and that $\zeta$ is Lipschitz implies
\ba{
\babs{\IE\cls{\zeta(F)}-\IE\cls{\zeta_{\eps,\delta}(F)}}
    &= \babs{\IE\cls{\zeta(F)}-\IE\cls{\zeta(F_\eps + \delta S)}}  \leq \IE \norm{F-F_\eps}_\infty + \delta \IE \norm{S}_\infty, 
}
where $S:\cS^n \to \IR^{d}$ is the smoothing Gaussian random field defined at~\eq{eq:smthcov}. Since $S$ has independent components and, by \eqref{eq:smthcov} and \eqref{eq:zkbds}, has covariance uniformly bounded in absolute value of order $c_n \sum_{k\geq 1} k^{-1-\iota}$ for some constant $c_n$ depending only on $n$, Fernique's theorem implies 
\be{
\IE \norm{S}_\infty\leq c_{n,\iota} \, \sqrt{d},
}
where $c_{n,\iota}$ constant depending only on $n$ and $\iota$. Thus, we find
\be{
\babs{\IE\cls{\zeta(F)}-\IE\cls{\zeta_{\eps,\delta}(F)}} \leq  \IE \norm{F-F_\eps}_\infty  + c_{n,\iota} \delta \sqrt{d}.
}
The same reasoning shows that this same inequality holds with $H$ replacing $F$. Substituting these bounds in~\eqref{eq:dwtin1} verifies that the desired bound in~\eqref{eq:masterresult} holds.
\end{proof}

\section{Properties of Solution to the Stein Equation}\label{sec:firstterm}

Applications of Theorem \ref{thm:strongermetric}
require bounds on the first three terms of the right-hand side of~\eqref{eq:masterresult}. In this section, we bound the first term for the case when $H=G$, the approximating Gaussian field. We handle the second and the third terms in the following section, see Lemma~\ref{lem:momepsreg} in particular. 

We start with the following result, which extends the work of \cite[Section~2]{Barbour2021a} and provides  properties of solutions to infinite-dimensional versions of Stein's equation. Specifically,~\cite{Barbour2021a} worked with Stein's equations for Gaussian processes indexed by an interval $[0,T]$, whereas here we work with random fields indexed by a compact measured metric space.

\begin{theorem}[Bounds on solutions of the Stein equation]\label{thm:stnsmoothbd}
For a Gaussian random field $G \in \mathrm{C}(\cM; \IR^d)$, define the operator $\cA=\cA_G$ acting on $\zeta:\mathrm{C}(\cM; \IR^d) \to \IR$ with 
\be{
\max_{k=1,2} 
\sup_{g \in \mathrm{C}(\cM; \IR^d)} 
\norm{D^k\zeta(g)}<\infty,
}
by
\be{
\cA \zeta(f):=\IE\bcls{D^2\zeta (f)[G,G]} - D \zeta (f)[f],
}
where $f\in \mathrm{C}(\cM; \IR^d)$, and $D$ denotes Frech\'et derivative.
Then for any such $\zeta$, there exists an $\eta=\eta_\zeta$ satisfying 
\ben{\label{eq:intco}
 \cA  \eta(f) =\zeta (f) - \IE[\zeta (G)].
}
Moreover, in the operator norm, for any $k=1,2,$ or $k\geq 3$ 
with $\sup_{g \in \mathrm{C}(\cM; \IR^d)} \norm{D^k \zeta(g)} <\infty$, we have
\begin{align}\label{eq:etad2lip}
\norm{D^k \eta(f)} \leq \frac{1}{k} \,\sup_{g\in \mathrm{C}(\cM; \IR^d)} \,\norm{D^k \zeta(g)}, 
\end{align}
and for any $\zeta\in \cF$ and all $f,h \in C(\mathcal{M};\mathbb{R}^d)$, we have
\ben{\label{eq:etad2is}
\norm{D^2 \eta(f)-D^2 \eta(h)} \leq \frac{1}{3} \norm{f-h}_\infty.
}
\end{theorem}
\begin{remark}\label{rem:stnmeth}
The operator $\mathcal{A}$ defined in Theorem \ref{thm:stnsmoothbd} plays the role of the left-hand side of the `random field' version 
\begin{align}\label{eq:process.Stein.equation}
\IE\bcls{D^2\eta (f)[G,G]} - D \eta (f)[f]=\zeta (f) - \IE[\zeta (G)]
\end{align}
of the finite dimensional Stein equation
for a centered Gaussian $G$ with covariance matrix \nr{$\Sigma=(\sigma_{ij})$ given by 
\be{
\sum_{i,j} \sigma_{ij} \partial_{ij}\eta(x) - \sum_{i} x_i \partial_i \eta(x) = \zeta(x)-\IE[\zeta(G)],
}
where $\partial_{i}$ and $\partial_{ij}$ denote the first and second partial derivatives respectively in coordinates $i,j$ .} With the Stein equation~\eqref{eq:process.Stein.equation} and the bounds on its solution provided by Theorem \ref{thm:stnsmoothbd}, the standard steps of Stein's method can be implemented. In particular, the integral probability metric bound to~$G$ over some given function class $\mathcal{H}$ can be computed by bounding the absolute expectation of the right-hand side of \eqref{eq:process.Stein.equation}
for $\zeta \in \mathcal{H}$ by taking absolute expectations on the left-hand side, given in terms of the solution~$\eta$.  In particular, uniformly bounding $\abs{\IE\cA  \eta_\zeta(F)}$ for all solutions $\eta_\zeta, \zeta \in \cF$ to \eqref{eq:process.Stein.equation} yields a bound on $\mathsf{d}_{\cF}(F,G)$, the first term on the right-hand side of~\eqref{eq:masterresult}. See Lemma \ref{lem:stnmethit} and its proof for the implementation.
\end{remark}
\begin{remark}
The term $\IE\bcls{D^2\zeta (f)[G,G]}$ implicitly depends on the covariance $C$ of $G$. As discussed in Remark~\ref{rem:stnmeth}, if $G$ is finite-dimensional, then this term evaluates explicitly to $\nabla^\top \Sigma \nabla \eta(x)$, where $\Sigma$ is the covariance matrix of $G$. When $G$ is a random field indexed by an uncountable set, in general it is not clear how to write this term solely in terms of $C$. In applications, the term should be rewritten in a form that matches the particular application and does not involve an expectation against $G$. Typically, this form involves the covariance structure of $G$ and is most easily found using some structure of the random field $F$ to determine the first order term in a Taylor expansion of $\IE \bcls{D \eta (F)[F]}$. See Section~\ref{sec:widednn} for further details from our application to wide random neural networks, and also \cite{Barbour2021a} for applications that provide additional relevant examples.
\end{remark}

\begin{proof}[\textbf{Proof of Theorem~\ref{thm:stnsmoothbd}}]
The result essentially follows from the work of \cite[Section~2]{Barbour2021a}, building off \cite{barbour1990stein} and \cite{Kasprzak2017}, in the setting where the index set
of the process $f$ is the interval $[0,T]$.

Fix $\zeta$ with two bounded derivatives. For $f\in \rC(\cM;\IR^d)$ define $h_f:\IR_+\to \IR$ by $h_f(t):= \IE [\zeta(e^{-t} f + \sqrt{1-e^{-2t}}G)]$, and $\eta=\eta_\zeta$ to be
\be{
\eta(f)= -\int_0^\infty\bclr{ h_f(t)- \IE \cls{\zeta(G)}} \mathrm{d}t.
}
The integral is well-defined since $\zeta$ has bounded derivative, $\norm{f}_\infty$ is finite by continuity and compactness of $\cM$,  and $\norm{G}_\infty$ has finite moments, by Gaussianity and path continuity.
That $\cA$ can be applied to~$\eta$ (meaning it has two bounded derivatives) follows essentially from \cite{barbour1990stein}, see also \cite{Kasprzak2017}, since dominated convergence implies that if $\sup_g \norm{D^k \zeta(g)} <\infty$, we have the well-defined expressions
\ben{\label{eq:derivex}
D^k \eta (f)[g_1,\ldots,g_k] =  -\int_0^\infty e^{-kt}\IE\bclc{D^k\zeta(e^{-t} f + \sqrt{1-e^{-2t}} G)[g_1,\ldots,g_k]}\mathrm{d}t,
}
from which the bounds in~\eqref{eq:etad2lip} easily follow.
For~\eqref{eq:etad2is}, if $\zeta\in\cF$ applying~\eqref{eq:derivex} with $k=2$ and \eqref{def:calF}, that assures elements of $\mathcal{F}$ to have a Lipschitz-1 second derivative, imply
\ba{
\babs{D^2 \eta(f)&[g,g]-D^2 \eta(h)[g,g]} \\
    &\leq \int_0^\infty e^{-2t}\IE\bclc{\babs{D^2\zeta(e^{-t} f + \sqrt{1-e^{-2t}} G)[g,g]-D^2\zeta(e^{-t} h + \sqrt{1-e^{-2t}} G)[g,g]}}\mathrm{d}t \\
    &\leq   \norm{g}_\infty^2 \int_0^\infty  e^{-2t} \norm{e^{-t}f-e^{-t} h}_\infty \mathrm{d}t \\
    &\leq \frac{1}{3}  \norm{f-h}_\infty \norm{g}_\infty^2.
}
We now show~\eqref{eq:intco}. 
We have
\ba{
\zeta (f) - \IE[\zeta (G)]&=-\int_0^\infty h_f'(t) \mathrm{d}t \\
	&= \int_0^\infty e^{-t} \IE\bcls{D\zeta(e^{-t} f + \sqrt{1-e^{-2t}}G)[f]} \mathrm{d}t \\
	&\qquad -\int_0^\infty \frac{e^{-2t}}{\sqrt{1-e^{-2t}}} \IE\bcls{D\zeta(e^{-t} f + \sqrt{1-e^{-2t}}G)[G]} \mathrm{d}t  \\
	& = -D\eta(f)[f] -\int_0^\infty \frac{e^{-2t}}{\sqrt{1-e^{-2t}}} \IE\bcls{D\zeta(e^{-t} f + \sqrt{1-e^{-2t}}G)[G]} \mathrm{d}t, 
	}
where the third equality follows by~\eq{eq:derivex}. Comparing to~\eq{eq:intco}, we 
will have shown the first claim if we can demonstrate that 
\be{
 \IE\bcls{D^2\eta (f)[G,G]}=-\int_0^\infty \frac{e^{-2t}}{\sqrt{1-e^{-2t}}} \IE\bcls{D\zeta(e^{-t} f + \sqrt{1-e^{-2t}}G)[G]} \mathrm{d}t.
}
Evaluating~\eq{eq:derivex} for $k=2$, the left-hand side of the previous display can be expressed as 
\be{
-\int_0^\infty e^{-2t} \IE\bcls{D^2\zeta(e^{-t} f + \sqrt{1-e^{-2t}} G)[G',G']} \mathrm{d}t,
}
where $G'$ is an independent copy of $G$.
But the integrands are equal since,  
for fixed $t$ and $f$ and $\phi(g)\coloneqq \zeta(e^{-t} f + \sqrt{1-e^{-2t}} g)$,~\cite[Proof of Proposition~2.1]{Barbour2021a} implies
\be{
\IE\bcls{D\phi(G)[G]} = \IE\bcls{D^2\phi(G)[G',G']},
}
and 
\ba{
D\phi(g)[g_1]&=\sqrt{1-e^{-2t}} D \zeta(e^{-t} f + \sqrt{1-e^{-2t}} g)[g_1],\\
D^2\phi(g)[g_1,g_2]&=(1-e^{-2t}) D^2 \zeta(e^{-t} f + \sqrt{1-e^{-2t}} g)[g_1,g_2]. 
}
Note that here we are using the Karhunen-Loeve expansion which is the part of the argument that uses the Borel measure~$\nu$ on our metric space $(\mathcal{M},\mathtt{d})$.
\end{proof}
There have been a number of recent works developing Stein's method for processes, predominantly in the context of distributional approximation by interval-indexed Gaussian processes, and especially Brownian motion; though see \cite{Gan2021} for an exception. 
Building from the seminal work of \cite{barbour1990stein},~\cite{shih2011steins} develops Stein's method in the very general setting of a Gaussian measure
on a separable Banach space. However, the bounds there are too abstract to be evaluated explicitly in practice. Closely following~\cite{shih2011steins}, the works~\cite{coutin2013stein,coutin2020stein,bourguin2020approximation}
provide more concrete results in the less general setting of a Gaussian measure on a Hilbert space. However, the associated probability metrics are with respect to the Hilbert space topology, e.g., $L^2$ and Sobolev, which are quite weak and do not see fundamental natural statistics such as finite dimensional distributions and extrema. The works of~\cite{kasprzak2020steina, kasprzak2020steinb,dobler2021stein}, based on \cite{barbour1990stein}, are more closely related to our work, but  work only in smooth function metrics like $\mathsf{d}_\cF$. We refer to~\cite[Section 1.1]{Barbour2021} for additional details and comparisons.

\section{Chaining arguments for modulus of continuity}\label{sec:2and3terms}
We now present results for bounding the second and third terms in~\eqref{eq:masterresult} that arise from the smoothing process. We start with a proposition that is useful for obtaining probabilistic bounds on the modulus of continuity of an $\IR^d$-valued random field on a compact metric space $(\cM,\dist)$.
\begin{definition}[Modulus of Continuity]\label{defn:modulus}
The modulus of continuity of a function $J: \cM \to \mathbb{R}^d$ at level $\theta>0$ is defined as
$\omega_J(\theta)\coloneqq\sup\bclc{\norm{J(x)-J(y)}_2: x,y \in \cM, \dist(x,y) < \theta}.$
\end{definition}

While the proofs below leverage standard chaining arguments, existing results seem not to provide the form of the results we require, as those mainly focus on expectation bounds and the case of $d=1$. 

Define the \emph{covering number} $\cN(\cM,\dist,\eps)$ (or just $\cN(\eps)$ when $(\cM,\dist)$ is clear from context) of $(\cM,\dist)$ at level $\eps>0$ as
the smallest cardinality over finite collections of points $U \subseteq \cM$ so that every point of $\cM$ is within $\eps$ of some point of $U$ (i.e., $U$ is an $\eps$-net).

\begin{proposition}\label{prop:modcontbd}
Let $(\cM,\dist)$ be a compact metric space and let $\clr{H(x)}_{x\in \cM}$ be an $\IR^d$-valued random field with continuous paths and write $H=(H_1,\ldots,H_d)$. Suppose there exist positive constants $c_0, \beta, \gamma$ and $c_1$ such that for any $x,y\in \cM$ and $i=1,\ldots, d$, we have
\ben{\label{eq:pairptbds}
\IP\bclr{\abs{H_i(x)-H_i(y)} \geq  \lambda} \leq c_0 \, \frac{\dist(x,y)^\beta}{\lambda^\gamma} \quad~\text{for all}~\lambda>0,}
and for every $\eps>0$ the covering numbers satisfy
\ben{ \label{eq:cov.num.alpha}
\cN(\eps)\leq c_1 \eps^{-\alpha}.
}
Then, if $\alpha<\beta/2$ there is a constant $c$ depending only on $\diam(\cM)$, $\alpha,\beta$, $\gamma, c_0, c_1$ such that for all $i=1,\ldots, d$ and $\theta>0$,
\ben{\label{eq:probbd2s}
\IP\bclr{\omega_{H_i}(\theta)> \lambda} \leq  c \,\frac{\theta^{\beta-2\alpha}}{\lambda^\gamma},
}
and for any $0<k< \gamma$,
 \ben{\label{eq:modcontmombd}
 \IE\bcls{\omega_H(\theta)^k} \leq  c \, d^{k/2} \, \theta^{k (\beta-2\alpha)/\gamma}.
 }
\end{proposition}
\begin{proof}

Following \cite[Chapter~VII, Section 2, 9 Chaining Lemma]{pollard2012convergence}, we can construct
a nested sequence of subsets $\cM_0\subseteq \cM_1\subseteq \cM_2\subseteq \cdots \subseteq \cM$ such that 
every $t\in \cM$ is within $\diam(\cM) 2^{-i}$ of a point of $\cM_i$, and 
\ben{\label{eq:setbd}
\abs{\cM_i} \leq \cN\bclr{\diam(\cM)2^{-(i+1)}}\leq c_1 \frac{2^{\alpha(i+1)}}{\diam(\cM)^\alpha}.
}
For each $x\in \cM$, there is a sequence $(x_i)_{i\geq1}$ with $x_i\in \cM_i$ and $\lim_{i\to\infty} x_i = x$; in particular $\cM^*:=\bigcup_i \cM_i$ is dense in $\cM$.

We first show that for all $\theta>0$,
\be{
\IP\clr{w_H(\theta)> \lambda}=\IP\clr{w^*_H(\theta) > \lambda},
}
where $w^*_H(\theta) = \sup\bclc{\norm{H(x)-H(y)}_2: x,y \in \cM^*, \dist(x,y) \leq \theta}$ is the modulus of continuity 
of $H:\cM^*\to \IR$ at level $\theta$, that is, only considering points in $\cM^*$. 
The equality holds as the events are equal. Clearly $\clc{w^*_H(\theta) > \lambda} \subseteq \clc{w_H(\theta) > \lambda}$ since $\cM^*\subseteq \cM$. For the other direction,
if there are $x,y\in \cM$ with $\nr{\dist(x,y)<\theta}$ and $\nr{\norm{H(x)-H(y)}_2 > \lambda}$, then letting $x_i, y_i \in \cM_i$ such that $x_i \to x$ and $y_i\to y$, then \nr{$\dist(x_i,y_i)\to \dist(x,y)<\theta$} and continuity of $H$ implies that \nr{$\norm{H(x_i)-H(y_i)}_2 \to \norm{H(x)-H(y)}_2 >\theta$} and so there must be some $i$ with \nr{$\dist(x_i,y_i)<\theta$} and \nr{$\norm{H(x_i)-H(y_i)}_2>\lambda$}.

To bound $w_H^*(\theta)$, let $\theta>0$ be fixed and let $x,y$  be arbitrary points in $\cM^*$ satisfying $\dist(x,y)<\theta$. Since the $\cM_i$'s are nested, there exists $n$ such that $x,y\in \cM_{n+1}$. Further, there are sequences $(x_i)_{i=0}^{n}, (y_i)_{i=0}^{n}$ such that $x_i,y_i\in \cM_i$, and $\dist(x_i,x_{i+1})\vee \dist(y_i,y_{i+1}) \leq \diam(\cM) 2^{-i+1}$, letting $x_{n+1}:=x$ and $y_{n+1}:=y$. These sequences can be constructed sequentially, e.g., set $x_{n}$ to be the nearest point in $\cM_{n}$ to  $x$, which must be within $\diam(\cM) 2^{-{n}}$ since $\cM_{n}$ is a $\diam(\cM)2^{-n}$-net. Given that we know $x_{i+1}$, we choose $x_i$ to be the point  in $\cM_{i+1}$ that is closest to $x_i$. Since $\cM_i \subseteq \cM_{i+1}$, there must be such a point with distance no greater than $\diam(\cM) 2^{-i+1}$. 

Denoting the maximum change in $H$ over points in $\cM_i$ that are within $\diam(\cM)\rho$ by  
\be{
D_{i}(\rho):=\sup\bclc{\norm{H(u)-H(v)}_2: u,v\in \cM_i, \, \dist(u,v) \leq \diam(\cM) \rho}.
}
Set $m=\lfloor - \log_2(\theta/{\rm diam}(\mathcal{M})) \rfloor$, implying in particular that $\theta \leq {\rm diam}(\mathcal{M})2^{-m}$. The triangle inequality implies that
\nr{
\ban{
\norm{H(x)-H(y)}_2& \leq \norm{H(x_m)-H(y_m)}_2 +  \sum_{i=m}^{n} \bbclr{\norm{H(x_{i+1}) - H(x_i)}_2  + \norm{H(y_{i+1}) - H(y_i)}_2} \notag \\
	&\leq  D_m\bclr{2^{-m+ 2}} + 2 \sum_{i=m}^{\infty}  D_i\bclr{2^{-i+1}},\label{eq:tnon}
}
}
where the $2^{-m+2}$ in the second inequality follows by the triangle inequality 
\ba{
\dist(x_m,y_m)& \leq  d(x,y) + \sum_{i=m}^{n}\dist(x_i,x_{i+1}) +  \sum_{i=m}^{n}\dist(y_i,y_{i+1}) < \theta + \sum_{i=m}^{n}\dist(x_i,x_{i+1}) +  \sum_{i=m}^{n}\dist(y_i,y_{i+1})\\
	&\leq  \diam(\cM) 2^{-m} + 2 \, \diam(\cM) \sum_{i=m}^{n} 2^{-i+1} 
	\leq 3\,\diam(\cM) 2^{-m}
	\leq \diam(\cM) 2^{-m+2}.
}
\nr{Noting~\eq{eq:tnon},} 
we set $\lambda_i=(1-a) a^{i-m}( \lambda/3)$  for $i\geq m$, for some $a\in(0,1)$ to be chosen later, and applying the union bound we have
\ben{ \label{eq:sum.for.whstar}
\IP\bclr{\omega_H^*(\theta)> \lambda} \leq \IP\bclr{D_m\bclr{2^{-m+ 2}} > \lambda/3}+  \sum_{i=m}^{\infty} \IP\bclr{D_i\bclr{2^{-i+ 1}} > \lambda_i}.
}
Now, again using a union bound,~\eq{eq:pairptbds} and~\eq{eq:setbd} yields that
\besn{
\IP\bclr{D_i(\rho)> \lambda}
&\leq \mathop{\sum_{u,v\in \cM_i}}_{ \dist(u,v)\leq \diam(\cM)\rho}\IP\bclr{\nr{\norm{H(u)-H(v)}_2}>\lambda} \\
&\leq c_0 \abs{\cM_i}^2 \frac{ \diam(\cM)^\beta \rho^\beta}{\lambda^\gamma} \notag \\
&\leq c_0 \, c_1^2\,   2^{2\alpha(i+1)}\frac{ \diam(\cM)^{\beta-2\alpha} \rho^\beta}{\lambda^\gamma},
}
and so for a constant $c'$ depending on $\diam(\cM)$, $\alpha,\beta$, $\gamma, c_0, c_1$, using that the first term in \eqref{eq:sum.for.whstar} can be bounded by a constant depending only on $\alpha, \beta,\gamma$ and $a$ times the bound on the first term of the sum that follows, we have
\be{
\IP\bclr{\omega_H^*(\theta)> \lambda} \leq c'  (1-a)^{-\gamma}  \lambda^{-\gamma}  \sum_{i=m}^\infty  \bbbclr{\frac{  2^{2\alpha-\beta} }{a^\gamma}}^i.
}
Since $2\alpha-\beta<0$, it is possible to choose $a$ such that $r:=2^{2\alpha - \beta}/a^\gamma<1$. So doing, we obtain
$$
\sum_{i=m}^\infty  \bbbclr{\frac{  2^{2\alpha-\beta} }{a^\gamma}}^i = (1-r)^{-1}r^m = (1-r)a^{-\gamma m}2^{(2\alpha - \beta)m} \leq (1-r)2^{(2\alpha - \beta)m},
$$
where we have used that $a \in (0,1)$ and is being raised to a positive power, so can be bounded by 1. 
Recalling that $m=\floor{-\log_2(\theta/\diam(\cM))}$, we hence observe that there is a constant $c$
depending on $\diam(\cM)$, $\alpha,\beta$, $\gamma, c_0, c_1$, such that
\be{
\IP\bclr{\omega_H^*(\theta)> \lambda} \leq  c \, \frac{\theta^{\beta-2\alpha}}{\lambda^\gamma}. 
}

Now, we proceed to prove~\eq{eq:modcontmombd} starting with the case $d=1$. Letting $\tilde c$ be a constant that may vary from line to line, but will at most only depend on   $\diam(\cM), \alpha, \beta,\gamma,c_0,c_1$, the result easily follows from~\eq{eq:probbd2s}, since under our hypotheses that $0<k<\gamma$ we have
\begin{align}\label{eq:dis1}
    \begin{aligned}
    \IE\bcls{\omega_F(\theta)^k} 
	&= \int_0^\infty \IP\bclr{\omega_F(\theta)>\lambda^{1/k}} \mathrm{Leb}(\mathrm{d}\lambda) \\
	&= \int_0^{\theta^{k(\beta-2\alpha)/\gamma}} \IP\bclr{\omega_F(\theta)>\lambda^{1/k}} \mathrm{Leb}(\mathrm{d}\lambda) + \int_{\theta^{k(\beta-2\alpha)/\gamma}}^\infty \IP\bclr{\omega_F(\theta)>\lambda^{1/k}} \mathrm{Leb}(\mathrm{d}\lambda) \\
    & \leq \theta^{k(\beta-2\alpha)/\gamma} + \tilde c \int_{\theta^{k(\beta-2\alpha)/\gamma}}^\infty \frac{\theta^{\beta-2\alpha}}{\lambda^{\gamma/k}} \mathrm{Leb}(\mathrm{d}\lambda)  \\
    &\leq \tilde c \, \theta^{k(\beta-2\alpha)/\gamma},
    \end{aligned}
\end{align}
as desired.

Now, for general $d\geq 1$, it is clear from the definition of the modulus of continuity that 
\ben{\label{eq:temp1}
\omega^2_F(\theta) \leq \sum_{i=1}^d \omega^2_{F_i}(\theta). 
}
Raising both sides of~\eqref{eq:temp1} to any positive power $k \geq 1$, and using that $(\sum_i a_i)^k \leq d^{k-1} \sum_i a_i^k$ for non-negative $a_i$, we have
\begin{align*}
\omega^{2k}_F(\theta) \leq \left(\sum_{i=1}^d \omega^2_{F_i}(\theta)\right)^k \leq d^{k-1} \sum_{i=1}^d \omega^{2k}_{F_i}(\theta).
\end{align*}
Taking expectation on both sides and applying~\eqref{eq:dis1}, we have 
\begin{align*}
\left( \mathbb{E}[\omega^{k}_F(\theta)]\right)^2\leq \mathbb{E}[\omega^{2k}_F(\theta)]  \leq d^{k-1}  \sum_{i=1}^d  \mathbb{E}[\omega^{2k}_{F_i}(\theta)] \leq \tilde{c} d^{k}  \, \theta^{2k(\beta-2\alpha)/\gamma},
\end{align*}
and taking square roots yields the desired inequality.
\end{proof}

\begin{lemma}\label{lem:momepsreg} Let $\cM=\cS^n \subset\IR^{n+1}$ for some  $n\geq 2$, with natural geodesic metric~$\dist$, and  $H=(H_1,\ldots,H_d):\cS^n \to \IR^d$ be a random field with  continuous paths, and $H_\eps$ be the $\eps$-regularization of $H$ defined at~\eq{eq:genrep} for a fixed $0<\eps<1$. If for all $i=1,\ldots, d$, for all $x,y\in \cS^n$, some constant $\hat c$, and some $p>n$ we have
\ben{\label{eq:distmombd}
\IE\bcls{\bclr{H_i(x)-H_i(y)}^{2p}} \leq \hat c \, \dist(x,y)^{2p},
}
then there is a constant $c$ depending only on $\hat c, n$, and $p$, such that
\be{
\IE \norm{H - H_\eps}_\infty \leq c\, \sqrt{d} \, \eps^{\frac{1}{2}(1-\frac{n}{p})}\sqrt{\log(1/\eps)}.
}
\end{lemma}
\begin{proof}
Using the alternative expression for $H_\eps$ given at~\eq{eq:alternative} in Proposition~\ref{prop:htker}, for any given $\theta>0$ we immediately have
\ba{
H(x)-H_\eps(x) &= \int_{y: d(x,y)\leq \theta} p(x,y;\eps) \bclr{H(x)-H(y)} \mathrm{d}y +\int_{y: d(x,y)> \theta} p(x,y;\eps) \bclr{H(x)-H(y)} \mathrm{d}y,
}  
where $\mathrm{d}y$ is the volume element on the sphere. It is easy to see that $\omega_{H}(\theta)$ is finite because $H$ is continuous and the sphere is compact. Hence, we can further bound
\ba{
\bnorm{H(x)-H_\eps(x)}_2
    &\leq \omega_{H}(\theta) +\sup_{u,v\in\cS^n} \bnorm{H(u)-H(v)}_2 \int_{y: d(x,y)> \theta} p(x,y;\eps)  \mathrm{d}y \\
    &= \omega_{H}(\theta) +\omega_{H}(\pi) \int_{y: d(x,y)> \theta} p(x,y;\eps)  \mathrm{d}y.
}
The heat kernel bounds of \cite[Theorem~1]{Nowak2019} imply
\be{
\int_{y: d(x,y)> \theta} p(x,y;\eps)  \mathrm{d}y \leq c_n \, e^{-\theta^2/(5\eps)},
} 
where $c_n$ is a constant depending only on $n$. Hence, we have that 
\ben{\label{eq:epregintbd}
\IE \norm{H-H_\eps}_\infty \leq \IE\bcls{\omega_{H}(\theta)} + c_n \IE\bcls{\omega_{H}(\pi)} e^{-\theta^2/(5\eps)}.
}
To bound $\IE\cls{\omega_H(\theta)}$ we apply Proposition~\ref{prop:modcontbd}, and  use Markov's inequality to find that
\be{
\IP\bclr{\abs{H_i(x)-H_i(y)} \geq \lambda} \leq \frac{\IE\bcls{\clr{H_i(x)-H_i(y)}^{2p}}}{\lambda^{2p}}.
}
Therefore~\eq{eq:pairptbds} is satisfied with $\beta=\gamma=2p$ and $c_0=\hat c$, due to our assumption~\eqref{eq:distmombd}.
To bound the covering numbers, standard volume arguments (see, for example, \cite[Corollary 4.2.13]{Vershynin2018}) imply that for all $\eps\in (0,1)$, we have 
\be{
\cN(\cS^n,\dist,\eps)\leq c_n \, \eps^{-n},
}
where $c_n$ is a constant depending only on $n$, thus \eqref{eq:cov.num.alpha} is satisfied with $\alpha=n$. Applying \eqref{eq:modcontmombd} of Proposition~\ref{prop:modcontbd} we find that there exists a constant $c$, whose value from line to line may change, but depends that only on $\hat c, n$, and $p$, such that
\be{
\IE\bcls{\omega_{H}(\theta)} \leq c \, \sqrt{d}\,  \theta^{1-\frac{n}{p}}.
}
Substituting this inequality in~\eq{eq:epregintbd} and setting $\theta=\sqrt{l \eps \log(1/\eps)}$ and  $l 
\geq 5(1-\frac{n}{p})/2$, 
we conclude that
\be{
\IE \norm{H-H_\eps}_\infty \leq c \, \sqrt{d} \, \eps^{\frac{1}{2}\clr{1-\frac{n}{p}}} \sqrt{\log\clr{1/\eps}}.
\qedhere
}
\end{proof}

\section{Proofs for wide random neural network approximations}\label{sec:widednn}
We now apply the general results developed in the previous sections to prove Theorems~\ref{thm:w1bound} and~\ref{thm:rateimprovment} on the smooth and Wasserstein distance bounds for wide random neural network. We follow the strategy based on induction as previously described in Section~\ref{sec:appdnn}. We first present the following result, obtained by applying Theorem~\ref{thm:stnsmoothbd} at a given, single layer of the network. One key element driving the result is the use of the classical Stein `leave-one-out' approach, see \eqref{eq:loo}.

\begin{lemma}\label{lem:stnmethit}
Let  $H:\cM\to \IR^m$ be a random field with continuous and i.i.d.\ coordinate 
 processes $H_1,\ldots,H_m$, and let $W:\IR^{m}\to \IR^{n}$ be an
$n\times m$ random matrix that is independent of $H$ and has centered independent entries having the same variance $\Var(W_{ij})=:c_w/m$, also satisfying $\IE[W_{ij}^4]\leq B (c_w/m)^2$, and $\sigma:\IR\to\IR$.
Define $F:\cM\to \IR^n$ by
\be{
F(x) = W \sigma\bclr{H(x)},
}
and assume $F\in L^2(\cM;\IR^n)$.
Let $G\in\rC(\cM;\IR^d)$ be a centered Gaussian random field with covariance function
\be{
C_{ij}(x,y):=\IE\bcls{F_i(x)F_j(y)}=\delta_{ij} c_w \IE\bcls{\sigma\bclr{H_1(x)}\sigma\bclr{H_1(y)}}.
}
Then for any $\zeta\in \cF$, we have
\ben{ \label{eq:mainbd.llem:stnmethit}
\babs{\IE\cls{\zeta(F)}-\IE\cls{\zeta(G)}}\leq  c_w^{3/2} B^{3/4} \IE\bcls{\norm{\sigma(H_1)}_\infty^3} \frac{n^{3/2}}{\sqrt{m}}.
}
\end{lemma}

\begin{proof}
We apply Theorem~\ref{thm:stnsmoothbd} with the Gaussian random field $G$ and $d=n$.  In particular, we obtain the bound \eqref{eq:mainbd.llem:stnmethit} by substituting $F$ for $f$ in \eqref{eq:process.Stein.equation}
and bounding the expectation of its right-hand side. 
Our first step is to derive a more useful representation for the second order term. We claim
\ben{ \label{eq:subD^2eta.for.A}
\IE\bcls{ D^2 \eta(f)[G, G]}= \IE\bcls{ D^2 \eta(f)[W\sigma(H),W\sigma(H)]}.
}
More generally, if the covariance of a 
centered Gaussian random field~$G\in\rC(\cM;\IR^d)$ satisfies
$C_{ij}(x,y)=\delta_{ij}\IE[G_i(x) G_i(y)]= \IE[R_i(x)R_j(y)]$ for  some centered $L^2(\cM;\IR^d)$ random field $R$ (not necessarily assumed Gaussian), then for any bilinear form $A$ with $\IE\bcls{A[G,G]}<\infty$, we have
$\IE\bcls{A[G,G]}=\IE\bcls{A[R,R]}$.

Equality~\eqref{eq:subD^2eta.for.A} is a consequence of the Karhunen-Loeve expansion; see e.g., \cite[Chapter~3]{adler2007random}, that states that there is an orthonormal basis $(\phi_k)_{k\geq1}$ of 
$L^2(\cM; \IR)$ and independent one dimensional centered Gaussian random variables  $(X_{ki})_{k\geq1, 1\leq i \leq d}$  with $\Var(X_{ki})=\lambda_{ki}>0$ such that 
$G_i = \sum_{k\geq1} X_{ki} \phi_k$, and the convergence is in~$L^2$. Since $R$ is also~$L^2$, we can expand $R_i= \sum_{k\geq1} Y_{ki} \phi_k$ with $Y_{ki} =\int R_i(x) \phi_k(x) \mathrm{d}x$, where $\mathrm{d}x$ is the volume measure associated to $\cM$, and the convergence is in~$L^2$.
Now, by linearity and that $\Cov(X_{ki}, X_{\ell j})=\delta_{ij} \delta_{k \ell} \lambda_{k i}$, we find 
\begin{align*}
\IE\bcls{A[G,G]}=\sum_{i=1}^{d}\sum_{k\geq 1} \lambda_{k i} A[\mathbf{e}_i\phi_{k i}, \mathbf{e}_i\phi_{k i}],
\end{align*} 
where $\mathbf{e}_i$ is the $d$-dimensional vector with a one in the $i^{\rm th}$ position, and zero elsewhere.

To show that we obtain the same quantity with $R$ replacing $G$, it is enough to show $\Cov(Y_{k i},Y_{\ell j})=\delta_{ij}\delta_{k\ell} \lambda_{k i}$. We use Mercer's theorem, which says that 
\be{
\IE[R_i(x)R_j(y)]=C_{ij}(x, y)=\delta_{ij} \sum_{m\geq1} \lambda_{m i} \phi_m(x) \phi_m(y),
}
where the convergence in the sum is uniform, and we obtain
\ba{
\IE[Y_{k i} Y_{\ell j}] &= \IE\bbbcls{\iint   R_i(x) R_j(y) \phi_k(x) \phi_\ell(y)\,\mathrm{d}x\,\mathrm{d}y} \\
	& = \iint  C_{ij}(x, y) \phi_k(x) \phi_\ell(y)\, \mathrm{d}x\,\mathrm{d}y =\delta_{ij} \sum_{m\geq1} \lambda_{m i} \iint    \phi_m(x) \phi_m(y) \phi_k(x)\phi_\ell(y)\, \mathrm{d}x\,\mathrm{d}y\\
	& = \delta_{ij}\sum_{m \geq 1}\lambda_{mi}\delta_{mk}\delta_{ml}
	=\delta_{ij}\delta_{k\ell} \lambda_{k i},
}
as $\phi_k, k \geq 1$ are orthonormal, thus proving claim \eqref{eq:subD^2eta.for.A}.

Let the pair $(\widehat W, \widehat H)$ be an independent copy of $(W, H)$. Clearly, the right-hand side of \eqref{eq:subD^2eta.for.A} is the same for both pairs and hence 
\ben{\label{eq:twobd}
|\IE \zeta(F) - \IE \zeta(G) |=\babs{\IE \bcls{D^2\eta(W\sigma (H) ) [\widehat W\sigma(\widehat H),\widehat W\sigma(\widehat H)] - D\eta(W \sigma(H)) [W\sigma(H)]}},
}
via \eqref{eq:process.Stein.equation} and independence. Hence, bounding the right-hand side of \eqref{eq:twobd} yields a bound on the left-hand side of \eqref{eq:process.Stein.equation}.

\ignore{By Theorem~\ref{thm:stnsmoothbd} and the fact that $\zeta\in\cF$, we have the key inequality 
\ben{\label{eq:etad2lip-delete.this.reference.in.previous.theorem}
\norm{D^2 \eta(f)-D^2 \eta(h)} \leq \tfrac{1}{3} \norm{f-h}_\infty.
}
} 
We first write
\be{
\widehat W\sigma(\widehat H) =\sum_{j=1}^{m}
\widehat{V}_j
\quad \mbox{where we set} \quad
\widehat{V}_j\coloneqq\sum_{i=1}^{n}\widehat W_{ij}\sigma(\widehat H_j)\mathbf{e}_i,
}
and adopt parallel notation to define $V_j$. Because $\widehat W_{ij}$ are independent of each other and of $W$, centered, and assumed to have common variance $c_w/m$, for the first term in \eq{eq:twobd} we have
\begin{align}\label{eq:covrep}
\IE \bcls{D^2\eta(W \sigma(H)) &[\widehat W\sigma(\widehat H),\widehat W \sigma(\widehat H)]}  = \sum_{j=1}^{m}  \IE\bbbclc{ D^2\eta(W \sigma(H)) [\widehat{V}_j,\widehat{V}_j]}.
\end{align}

Working now on the second term of \eq{eq:twobd},
\ben{ \label{eq:loo}
\bclr{W\sigma(H)}^{j}:=W\sigma(H)-V_j \quad \mbox{where} \quad V_j=\sum_{i=1}^{n}W_{ij}\sigma(H_j) \mathbf{e}_i,
}
and which is independent of $(W_{ij})_{i=1}^{n}$ and $H_j$. Using that independence to subtract a term with expectation zero in the second line below, followed by an application of a Taylor type argument, we have
\ban{
&\IE\bcls{D\eta(W\sigma(H)) [W \sigma(H)]} \notag \\
&= \sum_{j=1}^{m} \IE\bbbclc{D\eta(W \sigma(H))
 [V_j]-D\eta\bclr{(W \sigma(H))^{j}} [V_j]
 }
 \notag \\
	&=\sum_{j=1}^{m} \IE\bbbclc{ D^2\eta\bclr{(W \sigma(H))^{j}}[V_j, V_j]} 
 \notag \\ 
 &\quad+
 \sum_{j=1}^m\int_0^1 
 \IE\bbbclc{ D^2\eta\bclr{s W \sigma(H)+(1-s)(W \sigma(H))^{j}}
 [V_j, V_j]
    -D^2\eta\bclr{(W \sigma(H))^{j}}[V_j, V_j]
    }
    \mathrm{Leb}(\mathrm{d}s) \notag \\
	&=  \sum_{j=1}^{m} \IE\bbbclc{ D^2\eta\bclr{(W \sigma(H))^{j}}[\widehat{V}_j,\widehat{V}_j]}  \label{eq:2ordap} \\
 \begin{split}
	 &\quad+\sum_{j=1}^m\int_0^1 \IE\bbbclc{ D^2\eta\bclr{s W \sigma(H)+(1-s)(W \sigma(H))^{j}}[V_j,V_j] 
    -D^2\eta\bclr{(W \sigma(H))^{j}}[V_j,V_j]
    } \mathrm{Leb}(\mathrm{d}s).
 \end{split}\label{eq:3ordap}
 }

\ignore{
\ban{
&\IE\bcls{D\eta(W\sigma(H)) [W \sigma(H)]} \notag \\
	&= \sum_{j=1}^{m} \IE\bbbclc{D\eta(W \sigma(H))\bbbcls{\sum_{i=1}^{n}W_{ij}\sigma(H_j) \mathbf{e}_i}-D\eta\bclr{(W \sigma(H))^{j}}\bbbcls{\sum_{i=1}^{n}W_{ij}\sigma(H_j) \mathbf{e}_i}} \notag \\
	&=\sum_{j=1}^{m} \IE\bbbclc{ D^2\eta\bclr{(W \sigma(H))^{j}}\bbbcls{\sum_{i=1}^{n}W_{ij}\sigma(H_j) \mathbf{e}_i, \sum_{i=1}^{n}W_{ij}\sigma(H_j) \mathbf{e}_i}} \notag \\
 &\quad+\sum_{j=1}^m\int_0^1 \IE\bbbclc{ D^2\eta\bclr{s W \sigma(H)+(1-s)(W \sigma(H))^{j}}\bbbcls{\sum_{i=1}^{n}W_{ij}\sigma(H_j) \mathbf{e}_i, \sum_{i=1}^{n}W_{ij}\sigma(H_j) \mathbf{e}_i} \notag \\
    &\qquad\qquad\qquad\qquad -D^2\eta\bclr{(W \sigma(H)){j}}\bbbcls{\sum_{i=1}^{n}W_{ij}\sigma(H_j) \mathbf{e}_i, \sum_{i=1}^{n}W_{ij}\sigma(H_j) \mathbf{e}_i}} \mathrm{Leb}(\mathrm{d}s) \notag \\
	&= \sum_{j=1}^{m} \IE\bbbclc{ D^2\eta\bclr{(W \sigma(H))^{j}}\bbbcls{\sum_{i=1}^{n}\widehat W_{ij}\sigma(\widehat H_j) \mathbf{e}_i, \sum_{i=1}^{n}\widehat W_{ij}\sigma(\widehat H_j) \mathbf{e}_i}}  \label{eq:2ordap} \\
 \begin{split}
	 &\quad+\sum_{j=1}^m\int_0^1 \IE\bbbclc{ D^2\eta\bclr{s W \sigma(H)+(1-s)(W \sigma(H))^{j}}\bbbcls{\sum_{i=1}^{n}W_{ij}\sigma(H_j) \mathbf{e}_i, \sum_{i=1}^{n}W_{ij}\sigma(H_j) \mathbf{e}_i} \\
    &\qquad\qquad\qquad\qquad -D^2\eta\bclr{(W \sigma(H))^{j}}\bbbcls{\sum_{i=1}^{n}W_{ij}\sigma(H_j) \mathbf{e}_i, \sum_{i=1}^{n}W_{ij}\sigma(H_j) \mathbf{e}_i}} \mathrm{Leb}(\mathrm{d}s).
 \end{split}\label{eq:3ordap}
 }
} 
To bound~\eq{eq:twobd}, we first subtract this expression from~\eq{eq:covrep} and, then  bound the absolute value.
In particular, we first bound the absolute difference between~\eq{eq:covrep} and~\eq{eq:2ordap}, and then the
absolute value of~\eq{eq:3ordap}. For the former, applying the second inequality of~\eq{eq:etad2lip}, which gives that the second derivative of $\eta$ is Lipschitz, followed by H\"older's inequality, yields that this difference is bounded by
\begin{multline}
\sum_{j=1}^{m}\bbbabs{\IE\bbbclc{ D^2\eta\bclr{(W \sigma(H))^{j}}[\widehat{V}_j,\widehat{V}_j]
- \IE
D^2\eta\bclr{W \sigma(H)}[\widehat{V}_j,\widehat{V}_j]}} \\
	\leq \frac{1}{3}\sum_{j=1}^{m} \IE\bbbcls{
 \bbnorm{V_j
 }_\infty \bbnorm{\widehat{V}_j
 }_\infty^2
 } 
 \leq  
 \frac{1}{3} \sum_{j=1}^m \IE \bbbcls{
 \bbnorm{V_j
 }_\infty^3
 }.
 \label{eq:fh}
\end{multline}
Similarly, but more simply, the absolute value of~\eq{eq:3ordap} is bounded by one-half this same quantity.

\ignore{
\ban{
\sum_{j=1}^{m}&\bbbabs{ \IE\bbbclc{ D^2\eta\bclr{(W \sigma(H))^{j}}\bbbcls{\sum_{i=1}^{n}\widehat W_{ij}\sigma(\widehat H_j) \mathbf{e}_i, \sum_{i=1}^{n}\widehat W_{ij}\sigma(\widehat H_j) \mathbf{e}_i}} \notag \\
	&\hspace{3cm}-\sum_{j=1}^{m} \IE\bbbclc{ D^2\eta\bclr{W \sigma(H)}\bbbcls{\sum_{i=1}^{n}\widehat W_{ij}\sigma(\widehat H_j) \mathbf{e}_i, \sum_{i=1}^{n}\widehat W_{ij}\sigma(\widehat H_j) \mathbf{e}_i}}}\notag\\
	&\leq \frac{1}{3}\sum_{j=1}^{m} \IE\bbbcls{
 \sum_{i=1}^{n}\widehat W_{ij}\sigma(\widehat H_j) \mathbf{e}_i
 }_\infty^2 \bbnorm{\sum_{i=1}^{n} W_{ij}\sigma(H_j) \mathbf{e}_i}_\infty
 }. \label{eq:fh}
  } 

 
To bound~\eq{eq:fh}, we use the fact \nr{that $H_j$ is independent from} $W_{ij}, i=1,\ldots,n$, and again apply H\"older's inequality, to find that
\ba{
\frac{1}{3}\sum_{j=1}^{m} \IE\bbbcls{\bbnorm{\sum_{i=1}^{n} W_{ij}\sigma(H_j) \mathbf{e}_i}_\infty^3}
	& \leq \frac{1}{3} \sum_{j=1}^{m}\IE\bcls{\norm{\sigma(H_j)}_\infty^3}\IE \bbbcls{\bbnorm{\sum_{i=1}^{n} W_{ij}\mathbf{e}_i}^4}^{3/4} \\
	& = \frac{1}{3}\sum_{j=1}^{m}\IE\bcls{\norm{\sigma(H_j)}_\infty^3}\IE \bbbcls{\bbclr{\sum_{i=1}^{n} W_{ij}^2}^2}^{3/4} \\
	&\leq \frac{m}{3}\IE\bcls{\norm{\sigma(H_1)}_\infty^3} \bbbclr{\frac{n^2 B c_w^2}{m^2}}^{3/4} \\
	& =\frac{1}{3} c_w^{3/2} B^{3/4} \IE\bcls{\norm{\sigma(H_1)}_\infty^3} \frac{n^{3/2}}{\sqrt{m}},
	}
where we have used that 
 $\IE\bcls{W_{ij}^4}\leq B (c_w/m)^{2}$. Hence, we obtain the desired inequality,~\eqref{eq:mainbd.llem:stnmethit}. 
\end{proof}

In Section~\ref{sec:smoothmetric}, Lemma~\ref{lem:stnmethit} is used to derive bounds on the difference between $G\sp\ell$ and $F\sp\ell$ in the smooth function metric for general $(\cM,\dist)$. For informative bounds in the Wasserstein metric when $\cM\equiv \cS^n$, we apply Theorem~\ref{thm:strongermetric}. The following lemma gives some moment bounds that are used along with Lemma~\ref{lem:momepsreg} to bound the terms $\IE\norm{F-F_\eps}_\infty$ and $\IE\norm{G-G_\eps}_\infty$ appearing in~\eq{eq:masterresult}.

\begin{lemma}\label{lem:mcmombd}
For fixed $p\in \IN$, assume  $H \in L^{2p}(\cM;\IR^m)$ is a
random field with identically distributed coordinate processes,  and let
 $W:\IR^{m}\to \IR$ be an
$1\times m$ random matrix that is independent of $H$ and has centered independent entries satisfying 
$\IE[W_{1j}^{2p}]\leq \tilde c /m^p$,
for some $\tilde c>0$.
Letting $\sigma:\IR\to\IR$  be Lipschitz with constant $\Lip_\sigma$, 
define $F:\cM\to \IR$ by
\be{
F(x) = W \sigma\bclr{H(x)},
}
and finally, letting $A_{m}\sp{2p}$ be the set of $(j_1,\ldots,j_{2p})\in \{1,\ldots, m\}^{2p}$ where the label of every coordinate appears at least twice, 
there is a constant $c$ depending only on  $p$ and $\tilde c$ such that
\ban{
\IE\bcls{(F(x) - F(y))^{2p}}&
\leq 
 \frac{\tilde c}{m^p} \sum_{(j_1,\ldots,j_{2p})\in A_{m}\sp{2p}}\prod_{\ell=1}^{2p} \IE\bbbcls{\bbclr{\sigma(H_{j_\ell}(x))-\sigma\clr{H_{j_\ell}(y)}}^{2p}}^{1/(2p)} \label{eq:intlipmombd}\\
 &\leq c \, \Lip_\sigma^{2p} \IE \bcls{\bclr{H_1(x)-H_1(y)}^{2p}}.\label{eq:lipmombd}
}

\end{lemma}
\begin{proof}
For the first inequality,  direct calculation gives
\bes{
\IE\bcls{\clr{F(x)-F(y)}^{2p}}
	&= \sum_{j_1,\ldots, j_{2p}=1}^{n_1} \IE\bbbcls{ \prod_{\ell=1}^{2p} W_{1,j_\ell}}\IE\bbbcls{\prod_{\ell=1}^{2p} \bbclr{\sigma(H_{j_\ell}(x))-\sigma\clr{H_{j_\ell}(y)}}} \\
	&= \sum_{(j_1,\ldots,j_{2p})\in A_{m}\sp{2p}} \IE\bbbcls{ \prod_{\ell=1}^{2p} W_{1,j_\ell} }\IE\bbbcls{\prod_{\ell=1}^{2p} \bbclr{\sigma(H_{j_\ell}(x))-\sigma\clr{H_{j_\ell}(y)}}},
}
which follows since $W_{1j}$ are independent and have mean zero. From this~\eq{eq:intlipmombd} easily follows by H\"older's inequality.

As $H$ has identically distributed entries, we see that~\eq{eq:intlipmombd} satisfies, 
 \bes{
 \frac{\tilde c}{m^p} \sum_{(j_1,\ldots,j_{2p})\in A_{m}\sp{2p}}\prod_{\ell=1}^{2p}& \IE\bbbcls{\bbclr{\sigma(H_{j_\ell}(x))-\sigma\clr{H_{j_\ell}(y)}}^{2p}}^{1/(2p)} \\
	&=\frac{\tilde c}{m^p} \, \babs{A_{m}\sp{2p}}  \IE\bbbcls{\bbclr{\sigma(H_{1}(x))-\sigma\clr{H_{1}(y)}}^{2p}}, \\
	&\leq  c \,  \IE\bbbcls{\bbclr{\sigma(H_1(x))-\sigma\clr{H_{1}(y)}}^{2p}},
}
where the last inequality follows because $\babs{A_{m}\sp{2p}}=\bigo(m^p)$, with a constant depending only on $p$. The upper bound~\eq{eq:lipmombd} now easily follows, since $\sigma$ is Lipschitz.
\end{proof}

\subsection{$W_1$ bounds for wide random neural networks: Proof of Theorem~\ref{thm:w1bound}}

Combining the previous results, we can now prove our main theorem for wide random neural networks.
 \begin{proof}[\textbf{Proof of Theorem~\ref{thm:w1bound}}] 
 The proof proceeds by induction  on $\ell=2,\ldots, L$ for the hypotheses that there is a constant $c$ (which may change from line to line) depending only on $(c_w\sp{m},c_b\sp{m},B\sp{m})_{{m}=0}^L$, $n, p$ and $\sigma(0)$  such that
\besn{\label{eq:indhypdw}
&\mathsf{d}_\cW\bclr{F\sp{\ell},G\sp{\ell}}  \\
	&\leq c\nr{(1+ \Lip_\sigma^3)^{\ell-1}} \sum_{m=1}^{\ell-1}\bbbclr{n_{m+1}^{1/2}\bbbclr{\frac{n_{m+1}^4}{n_m}}^{(1-\frac{n}{p})/({8}(1-\frac{n}{p})+{6}(n+\iota))}\log(n_{m}/n_{m+1}^4)} \prod_{j=m+1}^{\ell-1} \IE \norm{W\sp j}_{\mathrm{op}}, 
}
and
\besn{\label{eq:fgmomind}
\IE\bcls{\clr{G_i\sp{\ell}(x)-G_i\sp{\ell}(y)}^{2p}} & \leq   c\, \nr{ \Lip_\sigma^{2p(\ell-1)}} \dist(x,y)^{2p}, \,\,\, i=1,\ldots, n_\ell,
}
and finally
\ben{\label{eq:siggellmombd}
\IE\bcls{\norm{\sigma(G_i\sp{\ell}- b_i\sp{\ell})}_\infty^3} \leq c\nr{\,(1 + \Lip_\sigma)^{3\ell}}, \,\,\, i=1,\ldots, n_\ell.
}

We first note that the bias $b^{(\ell)}$ plays no role in the bound and can be set to zero. The reduction is obvious for~\eqref{eq:fgmomind} and~\eqref{eq:siggellmombd}, since we can write $G^{(\ell)}=\widetilde G^{(\ell)} + b^{(\ell)}$, with $\widetilde{G}^{(\ell)}$ a Gaussian process independent of $b^{(\ell)}$, having covariance $\widetilde{C}^{(\ell)}(x,y)= C^{(\ell)}(x,y) - \rI_{n_2} c^{(\ell)}$.
 To see why we can also make this simplification for~\eqref{eq:indhypdw}, assume that this inequality holds for $F^{(\ell)}$ and $G^{(\ell)}$ when the biases are zero. Define ${\widetilde F}^{(\ell)}=F^{(\ell)}+b^{(\ell)}$ and ${\widetilde G}^{(\ell)}=G^{(\ell)}+b^{(\ell)}$, where the summands are independent. For any Lipschitz $\zeta:\rC(\cS^n;\IR^{n_2}) \to \IR$ we have, by independence, that 
 \begin{align*}
\babs{\IE\cls{\zeta(\widetilde{F}\sp \ell)}-\IE\cls{\zeta(\widetilde{G}\sp \ell)}}=
\babs{\IE\cls{\zeta(F^{(\ell)}+b^{(\ell)})-\zeta(G^{(\ell)}+b^{(\ell)})}}
=
\babs{\IE\cls{\widetilde \zeta(F\sp \ell)-\widetilde\zeta(G\sp \ell)}}
\end{align*}
where
\be{
\widetilde\zeta(f) = \IE\cls{\zeta(f+b\sp2)}, 
}
which is $1$-Lipschitz, since
\be{
\babs{\widetilde\zeta(f)-\widetilde\zeta(g)}= \babs{\IE\bcls{\zeta(f+ b\sp2)-\zeta(g+b\sp2)}} \leq \norm{f-g}_\infty.
}
Hence Wasserstein bounds in the case where the biases are non-zero are upper bounded by those in the zero bias case. 
Note that eliminating the biases $b\sp\ell$ in this manner requires them to be Gaussian, as otherwise the process $G\sp\ell$ may not be Gaussian.

\ignore{
 we want to bound
\be{
\babs{\IE\cls{\zeta(F\sp 2)}-\IE\cls{\zeta(G\sp 2)}},
}
where $\zeta:\rC(\cS^n;\IR^{n_2}) \to \IR$ is  such that $\abs{\zeta(f)-\zeta(g)}\leq \norm{f-g}_\infty$.
 Similar to $G\sp2$, we can decompose $F\sp{2}=\widetilde F\sp{2} + b\sp{2}$, where the summands are independent.
We can then define
\be{
\tilde\zeta(f) = \IE\cls{\zeta(f+b\sp2)},
}
which is $1$-Lipschitz, inherited from $\zeta$, since
\be{
\babs{\tilde\zeta(f)-\tilde\zeta(g)}= \babs{\IE\bcls{\zeta(f+ b\sp2)-\zeta(g+b\sp2)}} \leq \norm{f-g}_\infty.
}
Moreover, by independence, we have
\be{
\IE\cls{\zeta(F\sp2)} - \IE\cls{\zeta(G\sp2)} =\IE\cls{\tilde \zeta(\widetilde F\sp2)} - \IE\cls{\tilde \zeta(\widetilde G\sp2)}.
}
Thus it  is enough to bound the right hand for $1$-Lipschitz  $\tilde \zeta $, which is the same as assuming
$c_b\sp2=0$ in the original formulation, which we now assume. 
}

We now begin the proof of the base case, $\ell=2$.  We first show~\eqref{eq:fgmomind}, as well as some other related moment bounds used to show~\eqref{eq:indhypdw}. We start by applying~\eqref{eq:lipmombd} from Lemma~\ref{lem:mcmombd} with $W=W\sp1_{1, \cdot}, H=F\sp1$ and $m=n_1$, to find
\bes{
\IE\bbcls{\bclr{F\sp2_1(x)-F_1\sp2(y)}^{2p}}
    &\leq c \, \nr{\Lip_\sigma^{2p} }\IE\bbcls{\bclr{F\sp1_1(x)-F\sp1_1(y)}^{2p}}.
}
Applying~\eqref{eq:intlipmombd} from Lemma~\ref{lem:mcmombd} with $W=W\sp0_{1,\cdot}, H(x)=x$ and $m=n_0$, and \nr{$\sigma$ there equal to the identity}, we obtain
\ba{
 \IE\bbcls{\bclr{F\sp1_1(x)-F\sp1_1(y)}^{2p}} 
    \leq &~c \sum_{(j_1,\ldots,j_{2p})\in A_{n_0}\sp{2p}} \prod_{\ell=1}^{2p} \abs{x_{j_\ell}-y_{j_\ell}}	\leq  c \sum_{j_1,\ldots,j_{2p}=1}^{n_0}\prod_{\ell=1}^{2p} \abs{x_{j_\ell}- y_{j_\ell}} \\
	=&~c  \bbbclr{\sum_{j=1}^{n_0} \abs{x_{j}- y_{j}}}^{2p}	\leq  c  \norm{x-y}_2^{2p}  \leq c \, \dist(x,y)^{2p},
}
where the last inequality holds as $\norm{x-y}_2\leq \dist(x,y)$. Thus, we have shown
\ben{\label{eq:F2mcmombd}
\IE\bbcls{\bclr{F\sp2_1(x)-F_1\sp2(y)}^{2p}}\leq c \, \nr{\Lip_\sigma^{2p} }\dist(x,y)^{2p}.
}
Letting the variance of $F_1^{(2)}(x)-F_2^{(2)}(y)$ be denoted $\tau^2$, as the first and second moments of $G_1^{(2)}$ match those of $F_1^{(2)}$, we have in particular that $G_1^{(2)}(x)-G_2^{(2)}(y) \sim \mathcal{N}(0,\tau^2)$, and with $c_p=(2p-1) \times (2p-3) \ldots \times 3 \times 1$, using Jensen's inequality and \eqref{eq:F2mcmombd} we obtain
\begin{multline}\label{eq:G2mcmombd}
E\left[ (G_1^{(2)}(x)-G_2^{(2)}(y))^{2p}\right] = c_p \tau^{2p} \\=c_p\left(E[(F_1^{(2)}(x)-F_2^{(2)}(y))^2] \right)^p
\leq c_pE[(F_1^{(2)}(x)-F_2^{(2)}(y))^{2p}] \leq c \, \nr{\Lip_\sigma^{2p} }
 \dist(x,y)^{2p}.
\end{multline}
The same inequalities holds for all indices $i=1,\ldots, n_2$ since the coordinates all have the same distribution. Thus, we have established~\eqref{eq:fgmomind} for $\ell=2$. 

Now turning to~\eqref{eq:indhypdw}, 
 we bound
$\norm{F\sp2-F\sp2_\eps}_\infty, \norm{G\sp2-G\sp2_\eps}_\infty$, and $\mathsf{d}_{\cF}(F\sp2,G\sp2)$
and then invoke Theorem~\ref{thm:strongermetric}. 
Using \eqref{eq:F2mcmombd} and~\eqref{eq:G2mcmombd} in Lemma~\ref{lem:momepsreg} \nr{applied to $F\sp2/\Lip_\sigma$ and $G\sp2/\Lip_\sigma$} implies that 
\ben{\label{eq:f2g2epsbd}
\max\bclc{\IE\norm{F\sp2-F\sp2_\eps}_\infty, \IE\norm{G\sp2-G\sp2_\eps}_\infty }
\leq c \,\nr{\Lip_\sigma} \sqrt{n_2} \, \eps^{\frac{1}{2}\clr{1-\frac{n}{p}}} \sqrt{\log\clr{1/\eps}}.
}
The right-hand side of inequality~\eq{eq:regzetderiv} of Theorem~\ref{thm:regzet} with $k=2$ and $d=n_2$ gives an upper bound on the amount by which $\zeta_{\eps,\delta}$ needs to be scaled in order to satisfy the second derivative condition in 
\eqref{def:calF} and be an element of $\mathcal{F}$. Noting that $G\sp1$ is continuous with i.i.d.\ coordinate processes, we can apply 
Lemma~\ref{lem:stnmethit} with $F=F\sp2$, $H=G\sp1$, $n=n_2$ and $m=n_1$  to find
 \ben{\label{eq:f2g2smoothbd1}
 \babs{\IE\cls{\zeta_{\eps,\delta}(F\sp2)}-\IE\cls{\zeta_{\eps,\delta}(G\sp2)}}
  	\leq  c  \, \delta^{-2} \eps^{- 2 (n+\iota)}   \bclr{c_w\sp1}^{3/2} \bclr{B\sp1}^{3/4} \IE\bcls{\norm{\sigma(G_1\sp{1})}_\infty^3} \frac{n_{2}^{5/2}}{\sqrt{n_1}}.
	}
To bound $\IE\bcls{\norm{\sigma(G_1\sp{1})}_\infty^3}$, since $\sigma$ is Lipschitz, for a fixed $y\in \cS^n$ and any $x\in \cS^n$, we have
\bes{
\babs{\sigma(G_1\sp1(x))}
	&\leq \babs{\sigma(G_1\sp1(x))-\sigma(G_1\sp1(y))}+\abs{\sigma(G_1\sp1(y))-\sigma(0)} + \abs{\sigma(0)} \\
	&\leq \Lip_\sigma \bclr{\omega_{G_1\sp1}(\pi) + \abs{G_1\sp1(y)}} + \abs{\sigma(0)},
}
where $\omega_{G_1\sp1}(\theta)$ denotes the modulus of continuity of $G_1\sp1$ at level $\theta$; see Definition~\ref{defn:modulus}. Taking the supremum over $x$ implies
\besn{\label{eq:bdsupg1}
\bnorm{\sigma(G_1\sp{1})}_\infty \leq \nr{(\Lip_\sigma+1)}  \bclr{\omega_{G_1\sp1}(\pi) + \abs{G_1\sp1(y)} + \abs{\sigma(0)}}.
}
Because $G\sp1(y)=W\sp0 y$, it is easy to see that
\ben{\label{eq:3momnormbd}
\IE \bbcls{\bnorm{\sigma(G_1\sp{1})}_\infty^3} \leq \nr{(\Lip_\sigma+1)^3} .
}
Substituting this upper bound into~\eqref{eq:f2g2smoothbd1} and combining with~\eqref{eq:f2g2epsbd} in 
Theorem~\ref{thm:strongermetric} implies
\be{
\mathsf{d}_\cW(F\sp2, G\sp2)
    \leq c \nr{(\Lip_\sigma+1)^3}\sqrt{n_2}\bbclr{  \eps^{\frac{1}{2}\clr{1-\frac{n}{p}}} \sqrt{\log\clr{1/\eps}} + \delta +  \delta^{-2} \eps^{- 2(n+\iota)}    \frac{n_{2}^{2}}{\sqrt{n_1}}}.
}
Choosing 
\be{
\delta=\eps^{-\frac{2}{3}(n+\iota)}\bbbclr{\frac{ n_2^{4}}{n_1}}^{1/6} \,\, \mbox{ and } \,\,
\eps=\bbbclr{\frac{n_2^4}{n_1}}^{1/(3(1-\frac{n}{p})+4(n+\iota))} 
}
 we have shown that 
\be{
\mathsf{d}_\cW(F\sp2, G\sp2)
    \leq c \nr{(\Lip_\sigma+1)^3}\sqrt{n_2}\bbbclr{\frac{n_2^4}{n_1}}^{(1-\frac{n}{p})/(6(1-\frac{n}{p})+8(n+\iota))}\sqrt{\log(n_1/n_2^4)}.
}

For~\eqref{eq:siggellmombd}, in exactly the same way as~\eqref{eq:bdsupg1}, we have 
for any $y\in\cS^n$, 
\ben{\label{eq:g2sup3mom}
\babs{\sigma(G_1\sp2(x))}\leq \nr{(\Lip_\sigma+1)} \bclr{\omega_{G_1\sp2}(\pi) + \abs{G_1\sp2(y)} + \abs{\sigma(0)}}.
}
But~\eq{eq:fgmomind} and Proposition~\ref{prop:modcontbd} together  imply  \nr{(scaling by $\Lip_\sigma$)}
that
$\IE\cls{\omega_{G_1\sp2}(\pi) ^3}\leq c\nr{(\Lip_\sigma+1)^3}$. Because $G_1\sp2$ is Gaussian, we have that
\be{
\IE\bcls{\abs{G_1\sp2(y)}^3} = 2 \sqrt{2/\pi} \Var(G_1\sp2(y))^{3/2},
}
and, by definition and using~\eq{eq:3momnormbd}, 
\be{
\Var(G_1\sp2(y)) = c_w\sp{1} \IE\bcls{\sigma\bclr{G_1\sp{1}(y)}^2}+c_b\sp{1}\leq c \nr{(\Lip_\sigma+1)^2}.
}
Thus
\be{
\IE\bcls{\norm{\sigma(G_1\sp{2})}^3}\leq c \nr{(\Lip_\sigma+1)^6},
}
and the base case is established.

For the induction step, assume~\eqref{eq:indhypdw},~\eqref{eq:fgmomind}, and~\eqref{eq:siggellmombd} for some $\ell \geq 2$; we show these three conditions are satisfied when $\ell$ is replaced by $\ell+1$. For~\eqref{eq:fgmomind}, we have from the definition of the covariance $C\sp{\ell+1}$ of~$G\sp{\ell+1}$ that 
\bes{
\IE\bbcls{\bclr{G_1\sp{\ell+1}(x)-G_1\sp{\ell+1}(y)}^{2}}
     &= c \IE\bbbcls{\bbclr{\sigma\bclr{G_1\sp\ell(x)}-\sigma\bclr{G_1\sp\ell(y)}}^{2}} \\
    &\leq c \nr{\Lip_\sigma^2}\IE\bbcls{\bclr{G_1\sp\ell(x)-G_1\sp\ell(y)}^{2}}  \\
    &\leq c\, \nr{\Lip_\sigma^{2\ell}} \dist(x,y)^2,
}
where the first inequality uses that $\sigma$ is Lipschitz, and the second step the induction hypothesis. As $G^{(\ell+1)}$ is Gaussian, we now also have that
\ben{\label{eq:gellp1mombd}
\IE\bbcls{\bclr{G_1\sp{\ell+1}(x)-G_1\sp{\ell+1}(y)}^{2p}}\leq c \, \nr{\Lip_\sigma^{2\ell p}}\dist(x,y)^{2p},
}
thus advancing the induction hypothesis for~\eqref{eq:fgmomind}.

Now turning to~\eqref{eq:indhypdw}, we first define an intermediate random field 
\ben{\label{eq:intermediate}
\widehat F\sp{\ell+1}:=W\sp\ell \sigma \bclr{G\sp\ell},
}
where we take $G\sp\ell$ to be independent of $W\sp\ell$.
By the triangle inequality,  we have 
\ben{\label{eq:2terms}
\mathsf{d}_\cW\bclr{F\sp{\ell+1},G\sp{\ell+1})} \leq \mathsf{d}_\cW\bclr{F\sp{\ell+1},\widehat F\sp{\ell+1}}+ \mathsf{d}_\cW\bclr{\widehat F\sp{\ell+1}, G\sp{\ell+1}}.
}
By definition, for the first term
\besn{\label{eq:ter1}
\mathsf{d}_\cW\bclr{F\sp{\ell+1},\widehat F\sp{\ell+1}} 
	&= \mathsf{d}_\cW\bclr{W\sp\ell \sigma \clr{F\sp\ell}, W\sp\ell \sigma \clr{G\sp\ell}}.
}
The function $\widetilde \zeta (f) =\IE\bcls{ \zeta(W\sp\ell \sigma \clr{f})}$ satisfies 
\be{
\babs{\widetilde \zeta(f)- \widetilde \zeta(g)} \leq \IE\bcls{ \norm{W\sp\ell}_{\mathrm{op}} } \, \Lip_\sigma \norm{f-g}_\infty,
}  
and so the independence of $W\sp\ell$ from $F\sp\ell$ and $G\sp\ell$ implies~\eq{eq:ter1} is upper bounded as
\ben{\label{eq:t1ltbd}
\mathsf{d}_\cW\bclr{F\sp{\ell+1},\widehat F\sp{\ell+1}}\leq \IE\bcls{ \norm{W\sp\ell}_{\mathrm{op}} } \, \Lip_\sigma \mathsf{d}_\cW\bclr{F\sp{\ell}, G\sp{\ell}}.
}

Now working on the second term of~\eq{eq:2terms}, we apply Theorem~\ref{thm:strongermetric} and bound 
$\norm{\widehat F\sp{\ell+1}-\widehat F_\eps\sp{\ell+1}}_\infty$, $\norm{G\sp{\ell+1}-G_\eps\sp{\ell+1}}_\infty$, and $\mathsf{d}_{\cF}(F\sp{\ell+1}, G\sp{\ell+1})$.
By~\eqref{eq:lipmombd} of Lemma~\ref{lem:mcmombd} with $W=W_{1,\cdot}\sp\ell$, $H=G\sp\ell$ and $m=n_\ell$, we have
\be{
\IE\bbcls{\bclr{F_1\sp{\ell+1}(x)-F_1\sp{\ell+1}(y)}^{2p}}\leq c\, \nr{\Lip_\sigma^{2p}}
\IE\bbcls{\bclr{G_1\sp{\ell}(x)-G_1\sp{\ell}(y)}^{2p}}\leq c \,\nr{ \Lip_\sigma^{2p\ell}} \dist(x,y)^{2p},
}
where the last inequality holds via the induction hypothesis~\eqref{eq:fgmomind}.
In conjunction with inequality~\eq{eq:gellp1mombd} for $G\sp{\ell+1}$, 
Lemma~\ref{lem:momepsreg} \nr{(applied after scaling by $\Lip_\sigma^{\ell}$)} now implies that
\ben{\label{eq:fgellp1inf}
\max\bclc{\IE\norm{F\sp{\ell+1}-F\sp{\ell+1}_\eps}_\infty, \IE\norm{G\sp{\ell+1}-G\sp{\ell+1}_\eps}_\infty }
\leq c \, \nr{\Lip_\sigma^\ell}\sqrt{n_{\ell+1}} \, \eps^{\frac{1}{2}\clr{1-\frac{n}{p}}} \sqrt{\log\clr{1/\eps}}.
}
Now, Lemma~\ref{lem:stnmethit} with $F=\widehat F\sp{\ell+1}$ and $H=G\sp{\ell}$, noting that $G\sp{\ell}$ is continuous with i.i.d.\ coordinate processes,  
implies
\bes{
\mathsf{d}_{\cF}\bclr{\widehat F\sp{\ell+1}, G\sp{\ell+1}} 
    &\leq
\bclr{c_w\sp\ell}^{3/2} \bclr{B\sp\ell}^{3/4} \IE\bcls{\norm{\sigma(G_1\sp{\ell})}^3} \frac{n_{\ell+1}^{3/2}}{\sqrt{n_\ell}} \\
    &\leq\nr{(1+\Lip_\sigma)^{3\ell}} \bclr{c_w\sp\ell}^{3/2} \bclr{B\sp\ell}^{3/4}\frac{n_{\ell+1}^{3/2}}{\sqrt{n_\ell}},
}
where we have used the induction hypothesis~\eqref{eq:siggellmombd} in the final inequality. Applying this inequality along with~\eqref{eq:fgellp1inf} in Theorem~\ref{thm:strongermetric} yields
\besn{\label{eq:rateimprov1}
\mathsf{d}_\cW(\widehat F\sp{\ell+1}, G\sp{\ell+1}) &\leq c\nr{(1+\Lip_\sigma)^{3\ell}}  \sqrt{n_{\ell+1}} \bbbclr{
\delta^{-2} \eps^{- 2(n+\iota)}\frac{n_{\ell+1}^{2}}{\sqrt{n_\ell}} 
+\eps^{\frac{1}{2}\clr{1-\frac{n}{p}}} \sqrt{\log\clr{1/\eps}}+\delta},
}
and choosing 
\be{
\delta=\eps^{-\frac{2}{3}(n+\iota)}\bbbclr{\frac{ n_{\ell+1}^{4}}{n_\ell}}^{1/6} \,\, \mbox{ and } \,\,
\eps=\bbbclr{\frac{n_{\ell+1}^4}{n_\ell}}^{1/(3(1-\frac{n}{p})+4(n+\iota))} 
}
gives
\be{
\mathsf{d}_\cW(\widehat F\sp{\ell+1}, G\sp{\ell+1})
    \leq c\nr{(1+\Lip_\sigma)^{3\ell}} \sqrt{n_{\ell+1}}\bbbclr{\frac{n_{\ell+1}^4}{n_\ell}}^{(1-\frac{n}{p})/({6}(1-\frac{n}{p})+{8}(n+\iota))}\sqrt{\log(n_\ell/n_{\ell+1}^4)}.
}
Using this bound and~\eqref{eq:t1ltbd} in~\eqref{eq:2terms}, and applying the induction hypothesis~\eqref{eq:indhypdw} advances the induction for~\eqref{eq:indhypdw}.

Finally, advancing the induction for~\eqref{eq:siggellmombd}, i.e., bounding $\IE\bcls{\norm{\sigma(G_1\sp{\ell+1})}^3}\leq c\nr{\, (\Lip_\sigma+1)^{3(\ell+1)}}$, follows in exactly the same way as for the base case, starting at~\eqref{eq:g2sup3mom}.
 \end{proof}

\subsection{Improved $W_1$ bounds: Proof of Theorem~\ref{thm:rateimprovment}}\label{sec:smoothmetric}
This subsection proves  Theorem~\ref{thm:rateimprovment}, showing the rate improvement under the additional assumption that $\sigma$ has three bounded derivatives. The rate improvement illustrated in Remark~\ref{rem:rate.improvement} comes from the fact that for the induction steps, we work with the smooth metric $\mathsf{d}_{\cF}$, and only smooth at the final layer, rather than at each layer of the induction in the $\mathsf{d}_{\cW}$ metric; compare~\eqref{eq:rateimprov1} and~\eqref{eq:rateimprov2}.

\begin{theorem}\label{thm:nnsmoothbd}

Assume that $\sigma$ has three bounded derivatives, and let the weights satisfy the moment condition in~\eqref{eq:wellmomcond}. Recalling the definition of $\beta_L$ from~\eqref{eq:betaanda},  
for any $L\geq2$, there exists a positive constant $c$, depending only on $(c_w\sp\ell,c_b\sp\ell,B\sp\ell)_{\ell=0}^L, n, p$, and $\norm{\sigma\sp{k}}_\infty$, $k=1,2,3$, such that $\mathsf{d}_\cF(F\sp{L},G\sp{L}) \leq c \, \beta_L$.

\end{theorem}

\begin{proof}
The proof follows by an induction similar to that in the proof of Theorem~\ref{thm:w1bound}.
For the base case $L=2$, first note we can again set $b\sp2 =0$,
since if $\zeta\in \cF$, then straightforward considerations imply
$\widetilde \zeta(f):=\IE\cls{\zeta(f+b\sp2)}\in \cF$. Thus, for $\widetilde G\sp2-b\sp2$ and $\widetilde F\sp2=F\sp2-b\sp2$ we have 
\be{
\babs{\IE\cls{\zeta(F\sp2)}-\IE\bcls{\zeta(G\sp2)}}=\babs{\IE\cls{\widetilde \zeta(\widetilde F\sp2)}-\IE\bcls{\widetilde \zeta(\widetilde G\sp2)}},
}
and so it is enough to bound the right-hand side for generic $\widetilde \zeta \in \cF$.
With this simplification, we can apply 
Lemma~\ref{lem:stnmethit} with $m=n_1$ and $n=n_2$ to find
\be{
\mathsf{d}_{\cF}\bclr{F\sp2, G\sp2}\leq (c_w\sp1)^{3/2} (B\sp1)^{3/4} \IE\bcls{\norm{\sigma(G_1\sp1)}_\infty^3} \frac{n_2^{3/2}}{\sqrt{n_1}} \leq c  \frac{n_2^{3/2}}{\sqrt{n_1}},
}
where the last inequality follows from~\eqref{eq:3momnormbd}, which states  $\IE\bcls{\norm{\sigma(G_1\sp1)}_\infty^3}\leq c$.

To advance the induction, assume the bound on $\mathsf{d}_\cF(F\sp\ell, G\sp\ell)$.
In exactly the same way as above, we can assume $b\sp\ell=0$. 
Now, recall that in~\eqref{eq:intermediate}, we defined the intermediate random field 
\be{
\widehat F\sp{\ell+1}:=W\sp\ell \sigma \bclr{G\sp\ell},
}
where $G\sp\ell$ is independent of $W\sp\ell$.
The triangle inequality, as before, yields
\be{
\mathsf{d}_\cF\bclr{F\sp{\ell+1},G\sp{\ell+1})} \leq \mathsf{d}_\cF\bclr{F\sp{\ell+1},\widehat F\sp{\ell+1}}+ \mathsf{d}_\cF\bclr{\widehat F\sp{\ell+1}, G\sp{\ell+1}}.
}
and we again define the function $\widetilde \zeta (f) =\IE\bcls{ \zeta(W\sp\ell \sigma \clr{f})}$.
We need to argue that up to a constant factor, $\widetilde \zeta\in\cF$. Starting with the first derivative and denoting component-wise (Hadamard) multiplication by $\circ$, we first have
\ba{
\babs{\widetilde \zeta(f+g)-\IE\bcls{ \zeta(W\sp\ell\bclr{ \sigma \clr{f}+\sigma'(f)\circ g}}} 
    &\leq \sup_{h} \norm{D\zeta(h)} \IE \bnorm{W\sp\ell\bclr{\sigma(f+g)-\sigma(f)-\sigma'(f)\circ g}}\\
    &\leq \sup_{h} \norm{D\zeta(h)} \IE \norm{W\sp\ell}_{\mathrm{op}} \norm{\sigma''}_\infty \norm{g}_\infty^2.
}
Combining the above display with a direct Taylor-like computation, we next have that
\ba{
\widetilde \zeta (f+&g)- \widetilde \zeta(f) \\
   & =\IE \int_{[0,1]^2} D^2\zeta \bclr{W\sp\ell(\sigma(f)+s \, t \,\sigma'(f)\circ g)}\bcls{W\sp\ell \bclr{\sigma'(f)\circ g},W\sp\ell \bclr{\sigma'(f)\circ g}} \mathrm{Leb}(\mathrm{d}s,\mathrm{d}t) \\
   &\qquad + \IE \bbcls{D\zeta(W\sp\ell \sigma \clr{f})\bcls{W\sp\ell \bclr{\sigma'(f)\circ g}}} +\bigo\bclr{\norm{g}^2_\infty} \\
    &=\IE \bbcls{D\zeta(W\sp\ell \sigma \clr{f})\bcls{W\sp\ell \bclr{\sigma'(f)\circ g}}} +\bigo\bclr{\norm{g}^2_\infty},
}
so that
\be{
D\widetilde \zeta (f)[g]=\IE \bbcls{D\zeta(W\sp\ell \sigma \clr{f})\bcls{W\sp\ell \bclr{\sigma'(f)\circ g}}}.
}
Since $\sup_h \norm{D \zeta (h)}\leq 1$, it follows that
\be{
\sup_f \norm{D \widetilde \zeta(f)}\leq  \norm{\sigma'}_\infty \IE \norm{W\sp\ell}_{\mathrm{op}}.
}
Similar but more onerous computations show
\ba{
D^2 \widetilde \zeta(f)[g\sp1, g\sp2] 
    &= \IE \bbcls{D^2 \zeta(W\sp\ell\sigma(f))\bcls{W\sp\ell \bclr{\sigma'(f)\circ g\sp1},W\sp\ell \bclr{\sigma'(f)\circ g\sp2}} }\\
    &\qquad+ \IE \bbcls{D\zeta(W\sp\ell \sigma \clr{f})\bcls{W\sp\ell \bclr{\sigma''(f)\circ g\sp1 \circ g\sp2}}},
}
so that 
\be{
\sup_f \norm{D^2 \widetilde \zeta(f)}\leq  \norm{\sigma'}^2_\infty \IE \bbcls{ \norm{W\sp\ell}^2_{\mathrm{op}}}+\norm{\sigma''}_\infty \IE \norm{W\sp\ell}_{\mathrm{op}} <\infty.
}
Finally, some straightforward but space-consuming manipulations, using in particular that
\be{
\abs{D^2 \zeta(h)[g\sp1, g\sp2]}\leq 3 \norm{g\sp1}_\infty\norm{g\sp2}_\infty \norm{D^2 \zeta(h)},
}
from \cite[Lemma~2.4]{Barbour2021a}, imply that
\be{
\frac{\bnorm {D^2 \widetilde \zeta(h)-D^2 \widetilde \zeta(f)}}{\norm{f-h}}\leq c\, \max\bclc{1, \IE\bcls{\norm{W\sp\ell}_{\mathrm{op}}^3}}.
}
Hence, using the independence of $W\sp\ell$ with $F\sp\ell$ and $G\sp\ell$ we have
\be{
\mathsf{d}_\cF\bclr{F\sp{\ell+1},\widehat F\sp{\ell+1}}\leq c\, \max\bclc{1, \IE\bcls{\norm{W\sp\ell}_{\mathrm{op}}^3}}  \mathsf{d}_\cF\bclr{F\sp{\ell}, G\sp{\ell}},
}
and the proof now follows as that for Theorem~\ref{thm:w1bound},
mutatis mutandis.
\end{proof}

We are now ready to prove Theorem~\ref{thm:rateimprovment}, using Theorem~\ref{thm:strongermetric}. Compared to the proof of Theorem~\ref{thm:w1bound}, the specific choice of the smoothing and regularization terms, $\eps$ and $\delta$ are different, resulting in the required rate improvement.
\begin{proof}[\textbf{Proof of Theorem~\ref{thm:rateimprovment}}] 
    We apply Theorem~\ref{thm:strongermetric}, with $F=F\sp{L}$ and $W=G\sp{L}$, and hence with $d=n_L$.
    Applying Lemma~\ref{lem:momepsreg}, using induction with~\eqref{eq:lipmombd} and~\eqref{eq:F2mcmombd}, we have that
\bes{
\IE \norm{F\sp{L}- F\sp{L}_\eps}_\infty &\leq  c  \, \sqrt{n_{L}} \, \eps^{\frac{1}{2}\clr{1-\frac{n}{p}}} \sqrt{\log\clr{1/\eps}},
}
and the same bound also holds for $\IE \norm{G\sp{L}- G\sp{L}_\eps}_\infty$. 
From Theorem~\ref{thm:nnsmoothbd}, we have
\be{
\mathsf{d}_\cF(F\sp{L},G\sp{L})\leq c \, \beta_L.
}

Putting everything together, we have
\begin{align}\label{eq:rateimprov2}
\mathsf{d}_\cW(F\sp{L},G\sp{L})\leq    c \sqrt{n_L} \left( \sqrt{n_L}\, \beta_L \,\delta^{-2} \eps^{-2 (n+\iota)}  +  \eps^{\frac{1}{2}\clr{1-\frac{n}{p}}} \sqrt{\log\clr{1/\eps}}   + \delta \right).
\end{align}
Picking $\eps$ and $\delta$ as
\begin{align*}
   \delta =\eps^{-2(n+\iota)/3} (n_L\beta_L^2)^{1/6} \qquad \eps = (n_L\beta_L^2)^{1/(3(1-\frac{n}{p})+4(n+\iota))},  
\end{align*}
we obtain the desired result.  
\end{proof}

\bibliographystyle{abbrvnat}
\bibliography{ref}
\end{document}